\documentclass[a4paper]{article}

\usepackage[pages=all, color=black, position={current page.south}, placement=bottom, scale=1, opacity=1, vshift=5mm]{background}

\usepackage[margin=1in]{geometry} 

\usepackage{amsmath}
\usepackage{amsthm}
\usepackage{amssymb}

\usepackage[utf8]{inputenc}
\usepackage{hyperref}
\usepackage[normalem]{ulem}
\usepackage{graphicx}
\usepackage{verbatim,xcolor}
\usepackage{bm,caption}
\usepackage{multirow}
\usepackage{algorithmicx}
\usepackage[ruled]{algorithm}
\usepackage{algpseudocode,matlab-prettifier}
\usepackage[symbol]{footmisc}

\usepackage[sort&compress,numbers,square]{natbib}
\theoremstyle{plain}
\newtheorem{theorem}{Theorem}[section]

\newtheorem{lemma}[theorem]{Lemma}

\theoremstyle{definition}

\newtheorem{remark}[theorem]{Remark}

\graphicspath{{fig/}}
\usepackage{mathrsfs}

\usepackage{lipsum}
\numberwithin{equation}{section}

\usepackage{todonotes}

\newcommand{\ba}{{\bf a}}
\newcommand{\bb}{{\bf b}}

\newcommand{\be}{{\bf e}}
\newcommand{\bx}{{\bf x}}
\newcommand{\bn}{{\bf n}}
\newcommand{\bu}{{\bf u}}
\newcommand{\bv}{{\bf v}}

\newcommand{\bw}{{\bf w}}

\newcommand{\bV}{{\bf V}}

\newcommand{\bC}{{\bf C}}
\newcommand{\bD}{{\bf D}}

\newcommand{\bbf}{{\bf f}}

\newcommand{\Pih}{\Pi_h}
\newcommand{\Pz}{\mathcal{P}_0}

\newcommand{\cR}{\mathcal{R}}

\newcommand{\jump}[1]{[ #1 ]}
\newcommand{\avg}[1]{\{ #1\}}

\newcommand{\Honezd}{[H^1_0(\Omega)]^d}
 
 \newcommand{\Ltwoz}{L^2_0(\Omega)}
 \newcommand{\Vh}{\bV_h}
 
  \newcommand{\Ch}{\bC_h}
  \newcommand{\Dh}{\bD_h}
\newcommand{\Th}{\mathcal{T}_h}

\newcommand{\Eh}{\mathcal{E}_h}
\newcommand{\Eho}{\mathcal{E}_h^o}
\newcommand{\Ehb}{\mathcal{E}_h^b}

\newcommand{\norm}[1]{\lVert #1\rVert}
\newcommand{\enorm}[1]{\lVert #1\rVert_{\mathcal{E}}}
\newcommand{\snorm}[1]{|#1|}
\newcommand{\trinorm}[1]{{\left\vert\kern-0.25ex\left\vert\kern-0.25ex\left\vert #1 \right\vert\kern-0.25ex\right\vert\kern-0.25ex\right\vert}}

\newcommand{\buh}{\bu_h}
\newcommand{\ph}{p_h}

\title{Fully discrete analysis of pressure-robust enriched Galerkin methods for the time-dependent Stokes equations\footnote{Submitted to the editors in 2026.}}
\author{Seulip Lee,\thanks{Department of Mathematics, Tufts University, Medford, MA 02155 (\texttt{seulip.lee@tufts.edu})}
\and Lin Mu,\thanks{Department of Mathematics, University of Georgia, Athens, GA 30602 (\texttt{linmu@uga.edu})}}

\date{
}

\begin{document}
	\maketitle
	
	\begin{abstract}
This paper presents a fully discrete analysis of pressure-robust enriched Galerkin (EG) methods for the time-dependent Stokes equations. The spatial discretization employs a velocity reconstruction operator that preserves the continuous component of the EG space while mapping its enriched discontinuous component into the lowest-order Raviart-Thomas ($\mathcal{R}T_0$) subspace of $H(\mathrm{div})$, ensuring strict local mass conservation and pressure-robustness. Fully discrete schemes are constructed using the backward Euler and Crank-Nicolson methods, with the reconstruction incorporated into both the forcing and discrete time-derivative terms. We prove unconditional stability and derive optimal-order, parameter-explicit \textit{a priori} error estimates for both the velocity and pressure. In particular, the velocity estimates are independent of the continuous pressure and the irrotational component of the forcing term, contain no inverse-viscosity factors, and explicitly track the spatial mesh size, time-step size, and viscosity. Numerical experiments in two and three dimensions verify the predicted convergence rates, strict local mass conservation, and robustness with respect to the viscosity, while comparisons with non-pressure-robust variants demonstrate the importance of incorporating the reconstruction into the discrete time-derivative term.
\end{abstract}

\section{Introduction}

The time-dependent Stokes equations govern the behavior of viscous, incompressible fluid flows. Beyond their direct application in low-Reynolds-number regimes, such as microfluidics, lubrication theory, and biological transport, these equations serve as a fundamental model for analyzing and numerically simulating the full Navier-Stokes equations. In particular, the numerical solution of the time-dependent Stokes system is a core component of operator-splitting techniques, fractional-step methods, and Newton-like linearization strategies for high-Reynolds-number flows \cite{glowinski2003, quarteroni2000}. 
The design of spatial and temporal discretizations for this system requires numerical schemes that control spurious pressure effects, enforce the incompressibility constraint either strictly or weakly, and maintain stability under practical time-stepping procedures.

A primary difficulty for classical mixed finite element methods is the sensitivity of the discrete velocity error to the continuous pressure and the irrotational component of the forcing term \cite{linke2014}. In standard formulations, such as Taylor-Hood, mini-element, and Crouzeix-Raviart methods, the discrete velocity error may depend directly on the magnitude or derivatives of the continuous pressure. Consequently, when the continuous pressure is large or exhibits sharp gradients, the accuracy of the velocity approximation can degrade significantly, even when the exact velocity field is smooth or trivial. This structural dependence reflects a lack of pressure-robustness and may produce unphysical velocities or poor mass conservation in practical computations.

To address this limitation, several pressure-robust methodologies have been proposed. One approach employs grad-div stabilization to penalize violations of the velocity divergence constraint \cite{olshanskii2009}. Another uses $H(\mathrm{div})$-conforming finite element methods, such as Raviart-Thomas or Brezzi-Douglas-Marini spaces, within discontinuous Galerkin or modified continuous formulations to obtain pointwise divergence-free discrete velocities \cite{cockburn2007, neilan2014}.
More recently, Linke et al. introduced a velocity reconstruction technique that modifies the discrete forcing term by mapping test functions into $H(\mathrm{div})$-conforming spaces through local reconstruction operators, so that the resulting velocity error bounds are independent of the pressure \cite{lederer2017,linke2016}. Such a reconstruction technique has subsequently been extended to nonconforming and weak Galerkin methods \cite{linke2020,mu2020pressure}, virtual element methods \cite{frerichs2022divergence,wang2021pressure}, and related finite-element formulations \cite{lederer2017,lederer2020pressure,mu2021development}.

Enriched Galerkin (EG) methods combine features of continuous Galerkin (CG) and discontinuous Galerkin (DG) approaches \cite{sun2009}. By enriching a continuous finite-element space with piecewise constant or low-order discontinuous functions, EG formulations provide local mass conservation while maintaining a computational cost and degree-of-freedom (DoF) count comparable to CG techniques. An EG scheme using the minimal number of DoFs was developed for the steady-state Stokes equations \cite{YiEtAl22Stokes}. While pressure-robust EG formulations have been established for steady-state problems using velocity reconstruction strategies \cite{hu2024pressure, lee2024low}, their extension to fully discrete, time-dependent Stokes systems poses additional structural and analytical challenges.
Although fully discrete analyses of the time-dependent Stokes equations are available for various formulations \cite{bernardi2004posteriori,feng2021analysis,leykekhman2025fully,li2023error,lv2024pressure,poudel2025pressure}, a parameter-explicit pressure-robust analysis that incorporates velocity reconstruction into the discrete time-derivative term has not yet been established.

Therefore, this paper introduces and analyzes a pressure-robust, fully discrete EG formulation for the time-dependent Stokes equations. The key features and main contributions of this work are summarized as follows:
\begin{itemize}
    \item We propose a pressure-robust EG spatial discretization that achieves strict local mass conservation. The scheme uses a velocity reconstruction operator that preserves the continuous component of the EG basis functions and maps the enriched discontinuous component into the lowest-order Raviart-Thomas \((\mathcal{R}T_0)\) subspace of \(H(\mathrm{div})\).
    
    \item We develop fully discrete backward Euler and Crank-Nicolson schemes in which the velocity reconstruction is incorporated into both the forcing and discrete time-derivative terms. This construction is essential for retaining pressure-robustness at the fully discrete level.
    
    \item We prove unconditional stability and derive optimal-order, parameter-explicit \textit{a priori} error estimates for the velocity and pressure. Our analysis explicitly tracks the dependence on the spatial mesh size, time-step size, and viscosity, yielding velocity error estimates that are independent of the continuous pressure and the irrotational component of the forcing term and contain no inverse-viscosity factors.
    
    \item We present numerical experiments in two and three dimensions that verify the predicted convergence rates and demonstrate strict local mass conservation and robustness to the continuous pressure and the viscosity parameter. Comparisons with non-pressure-robust variants demonstrate the importance of incorporating the reconstruction into the discrete time-derivative term.
\end{itemize}

The remainder of this paper is organized as follows. Section~\ref{sec:preliminaries} introduces the mathematical notations, Sobolev spaces, and the enriched finite-element spaces. Section~\ref{sec:scheme} describes the semidiscrete EG formulation and contains the stability proofs. Section~\ref{sec:scheme-2} defines the fully discrete EG schemes for the backward Euler and Crank-Nicolson algorithms. In Section~\ref{sec:full_analysis}, we prove the comprehensive \textit{a priori} error estimates. Section~\ref{sec:numerical_experiments} presents the numerical verification benchmarks, and Section~\ref{sec:conclusion} provides concluding remarks.

\section{Preliminaries}\label{sec:preliminaries}

In this section, we introduce the model problem, notation, and finite-element setting used throughout the paper.

\paragraph{Model problem.}
Let $\Omega \subset \mathbb{R}^d$ ($d=2,3$) be a bounded polygonal or polyhedral domain with boundary $\partial\Omega$, and let $I=(t_0,t_f]$ denote the time interval. We seek the velocity field $\mathbf{u}: \Omega \times I \rightarrow \mathbb{R}^d$ and the pressure field $p: \Omega \times I \rightarrow \mathbb{R}$ satisfying
\begin{subequations}\label{sys: time_Stokes}
\begin{alignat}{2}
\frac{\partial \mathbf{u}}{\partial t} - \mu \Delta \mathbf{u} + \nabla p &= \mathbf{f} \quad && \text{in } \Omega \times I, \label{eqn: time_Stokes1} \\
\nabla \cdot \mathbf{u} &= 0 \quad && \text{in } \Omega \times I, \label{eqn: time_Stokes2} \\
\mathbf{u} &= \mathbf{0} \quad && \text{on } \partial\Omega \times I, \label{eqn: time_Stokes3} \\
\mathbf{u}(\cdot, t_0) &= \mathbf{u}^0 \quad && \text{in } \Omega, \label{eqn: time_Stokes4}
\end{alignat}
\end{subequations}
where $\mu > 0$ is the kinematic viscosity coefficient, $\mathbf{f}$ is a prescribed body force, and the initial velocity $\mathbf{u}^0$ is a solenoidal field satisfying the corresponding compatibility conditions.

\paragraph{Function spaces.} 
Let $\mathcal{D}\subset\mathbb{R}^d$, $d=2,3$, be a bounded Lipschitz domain. For $s\ge 0$, we denote by $H^s(\mathcal{D})$ the standard Sobolev space, equipped with norm $\|\cdot\|_{s,\mathcal{D}}$ and seminorm $|\cdot|_{s,\mathcal{D}}$.
In particular, $H^0(\mathcal{D})=L^2(\mathcal{D})$, with inner product $(\cdot,\cdot)_\mathcal{D}$. When $\mathcal{D}=\Omega$, the subscript $\mathcal{D}$ is omitted.
All definitions extend componentwise to vector- and tensor-valued functions.
We define the subspaces
\[
H_0^1(\mathcal{D}):=\left\{v\in H^1(\mathcal{D}):v|_{\partial \mathcal{D}}=0\right\},\qquad
L_0^2(\mathcal{D})
:=
\left\{v\in L^2(\mathcal{D}) : (v,1)_\mathcal{D}=0\right\},
\]
and denote by $P_k(\mathcal{D})$ the space of polynomials of degree at most $k$ on $\mathcal{D}$. The Hilbert space
\begin{equation*}
    H(\text{div},\mathcal{D}):=\left\{\bv\in [L^2(\mathcal{D})]^d:\mathrm{div}\;\bv\in L^2(\mathcal{D})\right\}
\end{equation*}
is equipped with the norm \(\norm{\bv}_{H(\mathrm{div},\mathcal{D})}^2:=\norm{\bv}_{0,\mathcal{D}}^2+\norm{\mathrm{div}\;\bv}_{0,\mathcal{D}}^2\).

\paragraph{Triangulation and mesh notation.}
Let $\Th$ be a shape-regular triangulation of $\Omega$ into triangles (if $d=2$) or tetrahedra (if $d=3$). We denote by $\Eh$ the set of all edges/faces of $\Th$, decomposed as
\(
\Eh=\Eho\cup\Ehb,
\)
where $\Eho$ and $\Ehb$ are the collections of interior and boundary edges/faces, respectively.
For each element $T\in\Th$, $h_T$ denotes its diameter and $\bn_T$ its outward unit normal vector on $\partial T$. For an interior edge/face $e\in\Eho$ shared by elements $T^+$ and $T^-$, $\bn_e$ denotes the unit normal pointing from $T^+$ to $T^-$. For $e\in\Ehb$, $\bn_e$ is the outward unit normal on $\partial\Omega$.

\paragraph{Broken spaces and mesh-dependent norms.}
Associated with $\Th$, we define the broken Sobolev space
\begin{equation*}
    H^s(\Th):=\left\{v\in L^2(\Omega):v|_T\in H^s(T),\ \forall T\in\Th\right\},
\end{equation*}
with norm
\begin{equation*}    \norm{v}_{s,\Th}:=\left(\sum_{T\in\Th} \norm{v}^2_{s,T}\right)^{1/2}.
\end{equation*}
For $s=0$, the associated inner product is denoted as $(\cdot,\cdot)_{\Th}$.
The corresponding piecewise polynomial space is
\begin{equation*}
    P_k(\Th) = \{v\in L^2(\Omega): v|_T\in P_k(T),\ \forall T\in\Th\}.
\end{equation*}
The $L^2$-inner product on $\Eh$ is written as $\langle\cdot,\cdot\rangle_{\Eh}$, with associated norm
\begin{equation*}
\norm{v}_{0,\Eh}:=\left(\sum_{e\in\Eh} \norm{v}^2_{0,e}\right)^{1/2}.
\end{equation*}

\paragraph{Jumps and averages.}
For $e\in \Eho$ shared by $T^+$ and $T^-$, let $v^\pm$ denote the trace of $v|_{T^\pm}$ on $e\in \partial T^+\cap \partial T^-$.
The jump and average operators are defined by
\begin{equation*}
    \jump{v}:=\left\{\begin{array}{cl}
        v^+-v^- & \text{on}\ e\in \Eho, \\
        v & \text{on}\ e\in\Ehb,
    \end{array}\right.
    \qquad
    \avg{v}:=\left\{\begin{array}{cl}
        (v^++v^-)/2 & \text{on}\ e\in \Eho, \\
        v & \text{on}\ e\in\Ehb.
    \end{array}\right.
\end{equation*}
 These definitions extend componentwise to vector- and tensor-valued functions.

\paragraph{Time-dependent setting.}
For a Banach space \(X\), we use the standard Bochner spaces
\(L^p(I;X)\) and \(H^1(I;X)\).
Given a uniform partition
\[
t_m=t_0+m\tau,
\qquad
m=0,1,\ldots,N,
\]
where \(\tau=(t_f-t_0)/N\), we write
\[
v^m:=v(\cdot,t_m)
\]
for the value of a time-dependent function \(v\) at the time level \(t=t_m\).
For sequences \(\{v^m\}_{m=0}^N\), the quantity
\[
\|v\|_{\ell^2(t_0,t_n;X)}^2:=\tau\sum_{m=0}^{n-1}\|v^{m+1}\|_X^2
\]
serves as a discrete analog of
\[
\|v\|_{L^2(t_0,t_n;X)}^2.
\]


\section{Semidiscrete formulation}\label{sec:scheme}

In this section, we introduce the enriched Galerkin (EG) finite-element spaces, which pair a piecewise linear velocity approximation with a piecewise constant pressure approximation, achieving local mass conservation with minimal additional degrees of freedom. Based on these spaces, we formulate a semidiscrete EG scheme for the time-dependent Stokes problem \eqref{sys: time_Stokes} and establish an energy stability estimate.

\subsection{Enriched Galerkin finite-element spaces}

We adopt the EG finite-element pair introduced in \cite{YiEtAl22Stokes}, which combines continuous piecewise linear Lagrange velocity functions enriched by a discontinuous, piecewise linear, and mean-zero vector function per element, together with a piecewise constant pressure space. This pair is designed for incompressible flow problems.

The continuous component of the velocity space is
\begin{equation*}
\Ch = \left\{\bv^C \in \Honezd : \bv^C|_{T} \in [P_1(T)]^d,\ \forall T \in \Th \right\},
\end{equation*}
and the discontinuous enrichment space is
\begin{equation*}
    \Dh = \left\{\bv^D \in [L^2(\Omega)]^d : \bv^D|_{T} = c_T (\bx - \bx_T),\ c_T \in \mathbb{R},\ \forall T \in \Th\right\},
\end{equation*}
where $\bx_T$ denotes the barycenter of $T$.
The EG velocity space is then
\begin{equation*}
    \Vh := \Ch \oplus \Dh,
\end{equation*}
so that every $\bv\in\Vh$ admits the unique decomposition
$\bv=\bv^C+\bv^D$ with $\bv^C\in\Ch$ and $\bv^D\in\Dh$.
The pressure space is
\begin{equation*}
    Q_h := \left\{ q \in \Ltwoz : q|_T \in P_0(T),\ \forall T \in \Th \right\}.
\end{equation*}

\subsection{Semidiscrete EG formulation}

We first introduce the interpolation operator $\Pih: [H^2(\Omega)]^d \to \Vh$ from \cite{yi2022locking}, defined by
\begin{equation}
\Pi_h\bw=\Pi_h^C\bw+\Pi_h^D\bw,\label{eqn: interpolation_operator}
\end{equation}
where $\Pi_h^C\bw\in \bC_h$
is the nodal interpolant of $\bw$, and $\Pi_h^D\bw\in \bD_h$ is determined elementwise by
$$(\nabla\cdot\Pi_h^D\bw,1)_T=(\nabla\cdot(\bw - \Pi_h^C \bw), 1)_{T},\quad\forall T\in \Th.$$
The semidiscrete EG formulation of \eqref{sys: time_Stokes} reads: find
\[\bu_h\in H^1(I;\Vh)\quad \text{and}\quad p_h\in L^2(I;Q_h),\]
with initial condition \(\bu_h(t_0)=\bu_h^0:=\Pi_h\bu^0\), such that, for a.e. \(t\in I\),
\begin{subequations}\label{sys: semi_eg}
\begin{alignat}{2}
\left(\frac{\partial \bu_h}{\partial t},\bv\right)+\mu\ba(\bu_h,\bv)  - \bb(\bv, \ph) &= (\bbf,  \bv), &&\quad \forall \bv \in\Vh, \label{eqn: semi_eg1}\\
\bb(\buh,q) &= 0, &&\quad \forall q\in Q_h.  \label{eqn: semi_eg2}
\end{alignat}
\end{subequations}
The bilinear forms $\ba(\cdot,\cdot)$ and $\bb(\cdot,\cdot)$ are defined by 
\begin{subequations}\label{sys: bilinear}
\begin{align}
\ba(\bv,\bw) &:= (\nabla \bv,\nabla \bw)_{\Th} -  \langle \avg{\nabla\bv} \bn_e, \jump{\bw} \rangle_{\Eh} \nonumber\\
&\qquad\qquad\qquad- \langle \avg {\nabla\bw} \bn_e,\jump{\bv}\rangle_{\Eh}  +  \rho \langle h_e^{-1}\jump{\bv},
\jump{\bw}\rangle_{\Eh}, \label{eqn: bia} \\
\bb(\bw,q)&:= (\nabla\cdot\bw, q)_{\Th} - \langle \jump{\bw}\cdot\bn_e,\avg{q} \rangle_{\Eh},  \label{eqn: bib}
\end{align}
\end{subequations}
where $\rho >0$ is a penalty parameter and $h_e = |e|^{1/(d-1)}$, with $|e|$ denoting the length (if $d=2$) or area (if $d=3$) of $e \in \Eh$.

Since the enrichment functions in $\Dh$ are discontinuous across element interfaces, an interior penalty discontinuous Galerkin (IPDG) approach is used to enforce continuity weakly in the above bilinear forms. A penalty-parameter-free alternative based on a weak Galerkin formulation was recently proposed in \cite{lee2024low}.

\subsection{Stability analysis}

Following \cite{YiEtAl22Stokes}, we equip $\Vh$ with the 
discrete $H^1$-norm
\begin{equation*}
    \enorm{\bv}^2 := \norm{\nabla \bv}_{0, \Th}^2 + \rho \norm{h_e^{-1/2}  \jump{\bv}}_{0, \Eh}^2.
\end{equation*}
For $\rho$ sufficiently large, the bilinear form $\ba(\cdot,\cdot)$ satisfies the following coercivity and continuity estimates \cite{YiEtAl22Stokes}: there exist positive constants $\kappa_1$ and $\kappa_2$, independent of $h$ and $\mu$, such that
\begin{alignat}{2}
\ba(\bv, \bv) & \ge \kappa_1 \enorm{\bv}^2, && \quad \forall \bv \in \Vh, \label{eqn: ba_coer} \\
|\ba(\bv, \bw)| & \leq \kappa_2 \enorm{\bv}\enorm{\bw}, && \quad \forall \bv, \bw \in \Vh. \label{eqn: ba_conti}
\end{alignat}

We now establish the energy stability of the semidiscrete scheme \eqref{sys: semi_eg}.

\begin{lemma}\label{lemma: semi_stability}
Let
$\bu_h\in H^1(I;\Vh)$ and 
$p_h\in L^2(I;Q_h)$
be the solution to \eqref{sys: semi_eg}. Then, for all \(t\in I\),
\begin{equation}\label{eqn: semi_stability}
\|\bu_h(t)\|_0^2
+
2\mu\kappa_1
\int_{t_0}^{t}
\enorm{\bu_h(s)}^2\,\mathrm{d}s
\le
e^{t-t_0}
\left(
\|\bu_h(t_0)\|_0^2
+
\int_{t_0}^{t}
\|\bbf(s)\|_0^2\,\mathrm{d}s
\right).
\end{equation}
\end{lemma}

\begin{proof}
Taking $\bv=\buh(t)$ in \eqref{eqn: semi_eg1} and
$q=\ph(t)$ in \eqref{eqn: semi_eg2}, and adding the two equations, we obtain
\[
\left(\frac{\partial\buh}{\partial t},\buh\right)
+
\mu\ba(\buh,\buh)
=
(\bbf,\buh).
\]
Since
\[
\left(\frac{\partial\buh}{\partial t},\buh\right)
=
\frac{1}{2}\frac{\mathrm{d}}{\mathrm{d}t}\|\buh\|_0^2,
\]
the coercivity estimate \eqref{eqn: ba_coer} gives 
\[
\frac{1}{2}\frac{\mathrm{d}}{\mathrm{d}t}\|\buh\|_0^2
+
\mu\kappa_1\enorm{\buh}^2
\le
(\bbf,\buh).
\]
Applying the Cauchy-Schwarz and Young's inequalities to the right-hand side yields
\[
(\bbf,\buh)
\le
\frac12\|\bbf\|_0^2
+
\frac12\|\buh\|_0^2,
\]
and hence
\[
\frac{\mathrm{d}}{\mathrm{d}t}\|\buh\|_0^2
+
2\mu\kappa_1\enorm{\buh}^2
\le
\|\bbf\|_0^2
+
\|\buh\|_0^2.
\]
Integrating over $(t_0,t)$ gives
\[
\|\buh(t)\|_0^2
+
2\mu\kappa_1
\int_{t_0}^{t}
\enorm{\buh(s)}^2\,\mathrm{d}s
\le
\|\buh(t_0)\|_0^2
+
\int_{t_0}^{t}
\|\bbf(s)\|_0^2\,\mathrm{d}s
+
\int_{t_0}^{t}
\|\buh(s)\|_0^2\,\mathrm{d}s.
\]
The estimate \eqref{eqn: semi_stability} then follows by Gr\"onwall's inequality.
\end{proof}

\section{Pressure-robust fully discrete formulations}\label{sec:scheme-2}

In this section, we present pressure-robust fully discrete EG formulations for the time-dependent Stokes problem~\eqref{sys: time_Stokes}. We first recall the velocity reconstruction operator introduced in \cite{hu2024pressure}, and then present backward Euler and Crank-Nicolson discretizations that incorporate it. These formulations serve as the basis for the fully discrete error analysis in Section~\ref{sec:full_analysis}, which establishes their pressure-robustness.

The velocity reconstruction operator \cite{hu2024pressure} is defined as $\cR: \Vh \to\mathcal{B}DM_1(\Th)\subset H(\text{div},\Omega)$ by
\begin{subequations}\label{sys: BDM}
\begin{alignat}{2}
\int_e (\cR \bv) \cdot\bn_e  p_1\  ds & = \int_e \avg{\bv}\cdot\bn_e p_1 \ ds,
 && \quad \forall p_1 \in P_1(e), \ \forall e \in \Eho,  \\
\int_e (\cR \bv) \cdot\bn_e  p_1\  ds & = 0,  && \quad \forall p_1 \in P_1(e), \ \forall e \in \Ehb,
\end{alignat}
\end{subequations}
where $\mathcal{B}DM_1(\Th)$ denotes the Brezzi-Douglas-Marini space of index one on $\Th$.
The reconstruction preserves the continuous component, $\cR\bv^C=\bv^C\in\mathcal{B}DM_1(\Th)$, while the discontinuous component is mapped into the lowest-order Raviart-Thomas space,
\[
\cR\bv^D\in \mathcal{R}T_0(\Th)\subset\mathcal{B}DM_1(\Th).
\]
Consequently, $\cR\bv = \bv^C + \cR\bv^D\in \mathcal{B}DM_1(\Th)$ for every $\bv\in\Vh$.

In \cite{hu2024pressure}, the reconstruction operator was applied only to the forcing term to achieve pressure-robustness for the stationary Stokes problem. For time-dependent problems, however, the divergence-free structure must be maintained throughout the temporal evolution. Motivated by this, we incorporate the reconstruction operator not only in the forcing term but also in the time-stepping. More precisely, the standard difference quotient
\[
\frac{\bu_h^{m+1}-\bu_h^m}{\tau}
\quad
\text{is replaced by}
\quad
\frac{\mathcal R\bu_h^{m+1}-\mathcal R\bu_h^m}{\tau},
\]
where $\bu_h^m$ approximates $\bu^m=\bu(\cdot,t_m)$.
As shown in the subsequent analysis, this modification is essential to obtain velocity error estimates that are free of pressure dependence and of inverse powers of $\mu$.

We first propose a fully discrete backward Euler formulation, which serves as the primary scheme analyzed in this paper and provides the basis for higher-order temporal discretizations.

\begin{algorithm}[H]
\caption{Fully pressure-robust backward Euler (\texttt{FPR-BE}) scheme} \label{alg: BE-FPR}
Given $\bu_h^m\in \Vh$, find $( \mathbf{u}_h^{m+1}, p_h^{m+1}) \in \mathbf{V}_h \times Q_h $ such that
\begin{subequations}\label{sys: BE-UR-EG}
\begin{alignat}{2}
\frac{1}{\tau}(\mathcal{R}\mathbf{u}_h^{m+1},\mathcal{R}\mathbf{v}) + \mu\mathbf{a}(\mathbf{u}_h^{m+1},\mathbf{v})  - \mathbf{b}(\mathbf{v}, p_h^{m+1}) &= (\mathbf{f}^{m+1},  \mathcal{R}\mathbf{v}) + \frac{1}{\tau}(\mathcal{R}\mathbf{u}_h^{m},\mathcal{R}\mathbf{v}), &&\quad \forall \mathbf{v} \in\mathbf{V}_h, \label{eqn: BE-UR-EG-mom}\\
\mathbf{b}(\mathbf{u}_h^{m+1},q) &= 0, &&\quad \forall q\in Q_h.  \label{eqn: BE-UR-EG-cont}
\end{alignat}
\end{subequations}
\end{algorithm}

The reconstruction strategy extends naturally to higher-order temporal discretizations. As an example, we consider the following Crank-Nicolson formulation, which achieves second-order time accuracy within the same reconstruction framework.

\begin{algorithm}[H]
\caption{Fully pressure-robust Crank-Nicolson (\texttt{FPR-CN}) scheme}\label{alg: CN-FPR}
Given $\bu_h^m\in\Vh$, find $(\bu_h^{m+1},p_h^{m+1/2})\in\Vh\times Q_h$ such that
\begin{subequations}\label{eqn: CN-UR-EG}
\begin{alignat}{2}
\frac{1}{\tau}(\mathcal{R}\bu_h^{m+1},\mathcal{R}\bv)
+ \frac{\mu}{2}\mathbf{a}(\bu_h^{m+1},\bv)
&- \mathbf{b}(\bv,p_h^{m+1/2})\nonumber\\
&= (\bar{\mathbf{f}}^{m+1/2},\mathcal{R}\bv)
+ \frac{1}{\tau}(\mathcal{R}\bu_h^{m},\mathcal{R}\bv)  - \frac{\mu}{2}\mathbf{a}(\bu_h^{m},\bv),
&&\quad\forall\bv\in\Vh, \label{eqn: CN-UR-EG-mom}\\
\mathbf{b}(\bu_h^{m+1},q) &= 0, &&\quad\forall q\in Q_h, \label{eqn: CN-UR-EG-cont}
\end{alignat}
\end{subequations}
where $\bar{\mathbf{f}}^{m+1/2}:=(\mathbf{f}^{m+1}+\mathbf{f}^m)/2$.
\end{algorithm}

In addition, both fully discrete formulations satisfy the following
mass-conservation property.
\begin{remark}[Mass conservation]\label{remark: mass_conservation}
At each time step, the reconstructed velocity is pointwise
divergence-free on every element, namely,
\[
\nabla\cdot\cR\bu_h^{m+1}=0
\quad \text{a.e. in each } T\in\Th.
\]
The proof is analogous to the mass-conservation argument in \cite{poudel2025pressure} and is therefore omitted here.
Since $\cR\bu_h^{m+1}\in H(\mathrm{div},\Omega)$, the reconstructed velocity is exactly locally mass-conservative at every time level. This property is verified numerically in Section~\ref{sec:numerical_experiments}.
\end{remark}

\section{Fully discrete error analysis}\label{sec:full_analysis}

In this section, we establish \textit{a priori} error estimates for the proposed fully discrete formulations. We begin by collecting several approximation and stability results used throughout the analysis, including properties of the EG interpolation operator, the local \(L^2\)-projection for the pressure, and the velocity reconstruction operator.

The interpolation operator $\Pi_h$ introduced in \eqref{eqn: interpolation_operator} satisfies the following approximation and stability estimates \cite{yi2022locking}:
\begin{subequations}\label{sys: Pih}
\begin{alignat}{2}
& |\bw - \Pih \bw | _{j,\Th} \leq C h^{m-j} |\bw|_{m},&&\quad 0 \leq j \leq m \leq 2, \quad\forall\bw\in[H^2(\Omega)]^d, \label{eqn: Pih_err} \\
& \enorm{\bw - \Pih \bw} \leq C h \norm{\bw}_2, &&\quad\forall \bw \in [H^2(\Omega)]^d, \label{eqn: Pih_energy_err}
\\
& \enorm{\Pih \bw} \leq C \snorm{\bw}_1,
&&\quad\forall \bw \in [H_0^1(\Omega)\cap H^2(\Omega)]^d. \label{eqn: Pih_stability}
\end{alignat}
\end{subequations}
For the pressure, the local $L^2$-projection $\mathcal{P}_0: H^1(\Omega) \to Q_h$, defined elementwise by
$$(q - \Pz q, 1)_T = 0,\quad\forall T\in\Th,$$
satisfies the approximation estimate
\begin{equation}
  \norm{ q - 
  \Pz q}_0 \leq C h \norm{q}_1,  \quad \forall q \in H^1(\Omega).  \label{eqn: P_err}  
\end{equation}
The velocity reconstruction operator \(\mathcal R\) defined in \eqref{sys: BDM} satisfies the approximation estimate \cite{hu2024pressure}:
\begin{equation}\label{eqn: cR-err}
\norm{ \bv - \cR\bv}_{0}\leq Ch\norm{h_e^{-1/2}\jump{\bv}}_{0,\Eh} \leq C h \enorm{\bv},\quad \forall\bv \in \Vh,
\end{equation}
and the stability estimate
    \begin{equation}
        \norm{\mathcal{R}\bv}_0\leq Ch\enorm{\bv},\quad\forall\bv\in\Vh,\label{eqn: norm_bound_RtoE}
    \end{equation}
where $C>0$ is independent of $\mu$ and $h$ in both cases.
The bound \eqref{eqn: norm_bound_RtoE} follows by combining \eqref{eqn: cR-err} with an estimate for $\norm{\bv}_0$ obtained from a scaling argument based on the piecewise linearity of $\bv$ on $\Th$; see \cite[Lemma~1]{lee2024uniform} for details.

Finally, we recall the discrete inf-sup condition \cite[Lemma~4.5]{YiEtAl22Stokes}: for
$\rho$ sufficiently large, there exists a constant
$\gamma>0$, independent of $\mu$ and $h$, such that
\begin{equation}
\inf_{q\in Q_h}\sup_{\bv\in\Vh}\frac{\bb(\bv,q)}{\enorm{\bv} \norm{q}_0}\geq \gamma.\label{eqn: inf-sup}
\end{equation}

\subsection{Error estimates for the FPR-BE method}

Throughout this subsection, we assume that the exact solution satisfies
\begin{equation}
    \bu\in L^\infty(I;[H_0^1(\Omega)\cap H^2(\Omega)]^d),
\qquad
p\in L^\infty(I;L_0^2(\Omega)\cap H^1(\Omega)),\label{eqn: up_regularity}
\end{equation}
and
\begin{equation}
    \partial_t\bu\in L^\infty(I;[H^1(\Omega)]^d),
\qquad
\partial_{tt}\bu\in L^2(I;[L^2(\Omega)]^d).\label{eqn: time_regularity}
\end{equation}
For each time level \(t_m\), we have \(\bu^m=\bu(\cdot,t_m)\in [H_0^1(\Omega)\cap H^2(\Omega)]^d\) and
\(p^m=p(\cdot,t_m)\in L_0^2(\Omega)\cap H^1(\Omega)\).
We introduce the standard error decompositions
\begin{equation*}
    \bm{\chi}_h^m:=\mathbf{u}^m-\Pi_h\mathbf{u}^m,\qquad\mathbf{e}_h^m:=\Pi_h\mathbf{u}^m-\mathbf{u}_h^m\in \Vh,
\end{equation*}
and
\begin{equation*}
    \xi_h^m:=p^m-\mathcal{P}_0 p^m,\qquad\epsilon_h^m:=\mathcal{P}_0 p^m-p_h^m\in Q_h,
\end{equation*}
so that $\bu^m-\bu^m_h = \bm{\chi}_h^m + \mathbf{e}_h^m$ and $p^m- p^m_h = \xi_h^m + \epsilon_h^m$.

We next derive the error equations for the \texttt{FPR-BE} scheme \eqref{sys: BE-UR-EG}.
\begin{lemma}\label{lemma: be_fpr_erreqn}
For any $\bv\in\Vh$ and $q\in Q_h$, the error components satisfy
\begin{subequations}\label{sys: ur_euler_erreqn}
\begin{alignat}{2}
\frac{1}{\tau}(\mathcal{R}(\mathbf{e}_h^{m+1}-\mathbf{e}_h^m),\mathcal{R}\mathbf{v})+\mu\mathbf{a}(\mathbf{e}_h^{m+1}&,\mathbf{v}) - \mathbf{b}(\mathbf{v},\epsilon_h^{m+1})\nonumber\\
&=l_1(\mathbf{u}^{m+1},\mathbf{v})+l_2(\mathbf{u}^{m+1},\mathbf{v}) - l^{m}_{\mathtt{BE}}(\mathbf{u},\mathbf{v}),\label{eqn: ur_euler_erreqn1}\\
\mathbf{b}(\mathbf{e}_h^{m+1},q) &= -\mathbf{b}(\bm{\chi}_h^{m+1},q),\label{eqn: ur_euler_erreqn2}
\end{alignat}
\end{subequations}
where 
\begin{align*}
    &l_1(\mathbf{u}^{m+1},\bv):=\mu\mathbf{a}(\Pi_h\mathbf{u}^{m+1}-\mathbf{u}^{m+1},\mathbf{v}),\\
    &l_2(\mathbf{u}^{m+1},\bv):=\mu(\Delta\mathbf{u}^{m+1},\mathcal{R}\mathbf{v}-\mathbf{v}),\\
    &l^{m}_{\mathtt{BE}}(\bu,\bv):=(\partial_t\mathbf{u}^{m+1},\mathcal{R}\mathbf{v})
-\frac{1}{\tau}(\mathcal{R}\Pi_h(\mathbf{u}^{m+1}-\mathbf{u}^m),\mathcal{R}\mathbf{v}),
\end{align*}
with $\partial_t \bu^{m+1}:=(\partial \bu/\partial t)(\cdot,t_{m+1})$ denoting the time derivative of $\bu$ evaluated at $t_{m+1}$.
\end{lemma}
\begin{proof}

We derive the two identities separately, beginning with \eqref{eqn: ur_euler_erreqn1}.
Recall from \cite{YiEtAl22Stokes} the identity $-(\Delta\bw,\bv)=\mathbf{a}(\bw,\bv)$, valid for any $\bw\in[H_0^1(\Omega)\cap H^2(\Omega)]^d$ and $\bv\in\Vh$. Writing
$\mathcal{R}\bv=\bv+(\mathcal{R}\bv-\bv)$ and applying this identity
to $\bu^{m+1}$ give
\[
-\mu(\Delta\mathbf{u}^{m+1},\mathcal{R}\mathbf{v})
= \mu\mathbf{a}(\Pi_h\mathbf{u}^{m+1},\mathbf{v})
-\mu\mathbf{a}(\Pi_h\mathbf{u}^{m+1}-\mathbf{u}^{m+1},\mathbf{v})
-\mu(\Delta\mathbf{u}^{m+1},\mathcal{R}\mathbf{v}-\mathbf{v}).
\]
For the pressure term, since $\mathcal{R}\bv\cdot\bn_T$
is continuous across element interfaces and $\nabla\cdot\mathcal{R}\bv$ is
piecewise constant on $\mathcal{T}_h$, elementwise integration by parts together with the continuity of $p^{m+1}$ gives
\[
(\nabla p^{m+1},\mathcal{R}\mathbf{v})
=-\mathbf{b}(\mathbf{v},\mathcal{P}_0 p^{m+1}).
\]
For the time derivative, adding and subtracting the discrete increment $\mathcal{R}\Pi_h(\bu^{m+1}-\bu^m)/\tau$ gives
\begin{align*}
(\partial_t\mathbf{u}^{m+1},\mathcal{R}\mathbf{v})
&= \frac{1}{\tau}(\mathcal{R}\Pi_h(\mathbf{u}^{m+1}-\mathbf{u}^m),\mathcal{R}\mathbf{v}) \\
&\quad\qquad +\left[(\partial_t\mathbf{u}^{m+1},\mathcal{R}\mathbf{v})
-\frac{1}{\tau}(\mathcal{R}\Pi_h(\mathbf{u}^{m+1}-\mathbf{u}^m),\mathcal{R}\mathbf{v})\right]\\
&= \frac{1}{\tau}(\mathcal{R}\Pi_h(\mathbf{u}^{m+1}-\mathbf{u}^m),\mathcal{R}\mathbf{v})
+ l^{m}_{\mathtt{BE}}(\bu,\bv),
\end{align*}
where $l^{m}_{\mathtt{BE}}$ captures the consistency error of the backward Euler discretization.
Testing \eqref{eqn: time_Stokes1} at $t=t_{m+1}$ against $\mathcal{R}\bv$ and substituting the three identities above yield
\begin{align*}
\frac{1}{\tau}(\mathcal{R}\Pi_h(\bu^{m+1}-\bu^m),\mathcal{R}\bv)
&+\mu\mathbf{a}(\Pi_h\bu^{m+1},\bv)-\mathbf{b}(\bv,\mathcal{P}_0 p^{m+1})\\
&=(\mathbf{f}^{m+1},\mathcal{R}\bv)
+ l_1(\bu^{m+1},\bv)+l_2(\bu^{m+1},\bv)-l^{m}_{\mathtt{BE}}(\bu,\bv).
\end{align*}
Subtracting the discrete momentum equation \eqref{eqn: BE-UR-EG-mom} of the \texttt{FPR-BE} scheme
and using the linearity of $\mathcal{R}$
yield \eqref{eqn: ur_euler_erreqn1}.

For \eqref{eqn: ur_euler_erreqn2}, since $\bu^{m+1}$ is divergence-free and $\bu^{m+1}_h$ satisfies the discrete continuity equation \eqref{eqn: BE-UR-EG-cont}, we have $\mathbf{b}(\mathbf{u}^{m+1},q)=0=\mathbf{b}(\mathbf{u}_h^{m+1},q)$ for all $q\in Q_h$. Therefore,
\[\mathbf{b}(\mathbf{e}_h^{m+1},q)
=\mathbf{b}(\Pi_h\bu^{m+1}-\bu^{m+1},q)
=-\mathbf{b}(\bm{\chi}_h^{m+1},q),\]
which proves \eqref{eqn: ur_euler_erreqn2}.
\end{proof}

\begin{lemma}\label{lemma: be_fpr_consistency}
For any $\bv\in \Vh$, the consistency terms satisfy
\begin{subequations}\label{sys: ur_suppest}
\begin{alignat}{2}
& |l_1(\bu^{m+1},\bv)|\leq C{\mu}h\norm{\bu^{m+1}}_2\enorm{\bv},\label{eqn: ur_suppest1} \\
& |l_2(\bu^{m+1},\bv)|\leq C{\mu}h\norm{\bu^{m+1}}_2\enorm{\bv},\label{eqn: ur_suppest2} \\
& |l^{m}_{\mathtt{BE}}(\bu,\bv)|\leq C\big(h\norm{\partial_t\bu}_{L^\infty(t_m,t_{m+1};[H^1(\Omega)]^d)}+\tau^{1/2}\norm{\partial_{tt}\bu}_{L^2(t_m,t_{m+1};[L^2(\Omega)]^d)}\big)\norm{\mathcal{R}\bv}_0,\label{eqn: ur_suppest3}
\end{alignat}
\end{subequations}
where $C$ is a generic positive constant independent of $\mu$ and $h$ and may vary in each case.
\end{lemma}
\begin{proof}
The estimates \eqref{eqn: ur_suppest1} and \eqref{eqn: ur_suppest2}
follow from the interpolation estimate \eqref{eqn: Pih_energy_err},
the continuity of \(\ba(\cdot,\cdot)\), and the approximation property
\eqref{eqn: cR-err}; see \cite{lee2024uniform} for the same argument.

It remains to bound \(l^{m}_{\mathtt{BE}}\). By definition,
\[
l^{m}_{\mathtt{BE}}(\bu,\bv)
=
\left(
\partial_t\bu^{m+1}
-
\frac{\mathcal R\Pi_h(\bu^{m+1}-\bu^m)}{\tau},
\mathcal R\bv
\right).
\]
Adding and subtracting
\(
(\bu^{m+1}-\bu^m)/\tau
\), we split
\begin{align*}
l^{m}_{\mathtt{BE}}(\bu,\bv)
&=
\left(
\partial_t\bu^{m+1}
-
\frac{\bu^{m+1}-\bu^m}{\tau},
\mathcal R\bv
\right)\\
&\qquad\qquad
+
\left(
\frac{\bu^{m+1}-\bu^m}{\tau}
-\frac{\mathcal R\Pi_h(\bu^{m+1}-\bu^m)}{\tau},
\mathcal R\bv
\right)
\\
&=: \mathbf{I}_1+\mathbf{I}_2.
\end{align*}
For \(\mathbf{I}_1\), Taylor's theorem with integral remainder gives
\[
\left\|
\partial_t\bu^{m+1}
-
\frac{\bu^{m+1}-\bu^m}{\tau}
\right\|_0
\le
C\tau^{1/2}\norm{\partial_{tt}\bu}_{L^2(t_m,t_{m+1};[L^2(\Omega)]^d)},
\]
and hence
\[
|\mathbf{I}_1|
\le
C\tau^{1/2}\norm{\partial_{tt}\bu}_{L^2(t_m,t_{m+1};[L^2(\Omega)]^d)}
\|\mathcal R\bv\|_0.
\]
For \(\mathbf{I}_2\), setting $\bm\omega:=(\bu^{m+1} -\bu^m)/\tau$ and applying the triangle inequality give
\[
\|(\mathcal{I}-\mathcal R\Pi_h)\bm{\omega}\|_0
\le
\|\bm{\omega}-\Pi_h \bm{\omega}\|_0
+
\|\Pi_h \bm{\omega}-\mathcal R\Pi_h \bm{\omega}\|_0.
\]
The interpolation estimate \eqref{eqn: Pih_err}, the approximation property \eqref{eqn: cR-err}, and the stability estimate \eqref{eqn: Pih_stability} together yield
\[
\|(\mathcal{I}-\mathcal R\Pi_h)\bm{\omega}\|_0
\le
Ch\|\bm{\omega}\|_1.
\]
Since
\[
\bm\omega
=
\frac{1}{\tau}\int_{t_m}^{t_{m+1}}\partial_t\bu(s)\,\mathrm{d}s,
\]
the assumed regularity \eqref{eqn: time_regularity} gives
\[
\|\bm\omega\|_1
\le
\|\partial_t\bu\|_{L^\infty(t_m,t_{m+1};[H^1(\Omega)]^d)},
\]
and hence
\[
|\mathbf{I}_2|
\le
Ch\|\partial_t\bu\|_{L^\infty(t_m,t_{m+1};[H^1(\Omega)]^d)}
\|\mathcal R\bv\|_0.
\]
Combining the estimates for \(\mathbf{I}_1\) and \(\mathbf{I}_2\) proves
\eqref{eqn: ur_suppest3}.
\end{proof}

\begin{lemma} For any $\bv\in\Vh$ and $q\in Q_h$,
    \begin{subequations}\label{sys: supp_conti}
\begin{alignat}{2}
& |\bb(\bv,\xi_h^{m+1})|\leq Ch \norm{p^{m+1}}_1\enorm{\bv},\label{eqn: supp_conti1} \\
& |\bb(\bm{\chi}_h^{m+1},q)|\leq Ch\norm{q}_0\norm{\bu^{m+1}}_2,\label{eqn: supp_conti2}
\end{alignat}
\end{subequations}
where \(C>0\) independent of \(\mu\) and \(h\).
\end{lemma}
\begin{proof}
The estimates follow directly from the arguments in the proof of
Lemma 7 in \cite{lee2024uniform}, applied at the time level
\(t_{m+1}\).
\end{proof}

\begin{theorem}\label{thm: euler_error_estimate}
Under the regularity assumptions
\eqref{eqn: up_regularity}--\eqref{eqn: time_regularity},
for any integer \(1\le n\le N\), the following estimate holds:
\begin{align*}
    \|\mathcal{R}\mathbf{e}_h^n\|_0^2 + C\mu\norm{\be_h}^2_{\ell^2(t_0,t_n;\mathcal{E})} 
    &\leq C\|\mathcal{R}\mathbf{e}_h^0\|_0^2 + C\!\left(\frac{h^4}{\tau^2}+\mu h^2+h^4\right)\!\norm{\bu}^2_{\ell^2(t_0,t_n;[H^2(\Omega)]^d)}\\
    &\qquad\quad+Ch^2\norm{\partial_t\bu}_{L^\infty(t_0,t_n;[H^1(\Omega)]^d)}^2+C\tau^2\norm{\partial_{tt}\bu}_{L^2(t_0,t_{n};[L^2(\Omega)]^d)}^2.
\end{align*}
\end{theorem}
\begin{proof}
    We first rewrite the error equation \eqref{eqn: ur_euler_erreqn1} as
    \begin{align*}
        \mathbf{b}(\mathbf{v},\epsilon_h^{m+1}) &= \frac{1}{\tau}(\mathcal{R}(\mathbf{e}_h^{m+1}-\mathbf{e}_h^m),\mathcal{R}\mathbf{v})_{\mathcal{T}_h}+\mu\mathbf{a}(\mathbf{e}_h^{m+1},\mathbf{v}) -l_1(\mathbf{u}^{m+1},\mathbf{v})-l_2(\mathbf{u}^{m+1},\mathbf{v}) + l^{m}_{\mathtt{BE}}(\mathbf{u},\mathbf{v}).
    \end{align*}
    It follows from the Cauchy-Schwarz inequality,
the continuity estimate \eqref{eqn: ba_conti},
and the supplemental estimates \eqref{sys: ur_suppest}
that
    \begin{align*}
        |\mathbf{b}(\mathbf{v},\epsilon_h^{m+1})|&\leq \frac{1}{\tau}\norm{\mathcal{R}(\mathbf{e}_h^{m+1}-\mathbf{e}_h^m)}_0\norm{\mathcal{R}\mathbf{v}}_0+\mu\kappa_2\enorm{\mathbf{e}_h^{m+1}}\enorm{\mathbf{v}}+ C\mu h\norm{\bu^{m+1}}_2\enorm{\bv} \\
        &\qquad\quad +Ch\norm{\partial_t\bu}_{L^\infty(t_m,t_{m+1};[H^1(\Omega)]^d)}\norm{\mathcal{R}\bv}_0+ C\tau^{1/2}\norm{\partial_{tt}\bu}_{L^2(t_m,t_{m+1};[L^2(\Omega)]^d)}\norm{\mathcal{R}\bv}_0.
    \end{align*}
    Applying the discrete inf-sup condition \eqref{eqn: inf-sup} together with the norm bound \eqref{eqn: norm_bound_RtoE}, we obtain
    \begin{align*}
        \gamma\enorm{\bv}\norm{\epsilon_h^{m+1}}_0&\leq C\frac{h}{\tau}\norm{\mathcal{R}(\mathbf{e}_h^{m+1}-\mathbf{e}_h^m)}_0\enorm{\mathbf{v}}+\mu\kappa_2\enorm{\mathbf{e}_h^{m+1}}\enorm{\mathbf{v}}+C\mu h\norm{\bu^{m+1}}_2\enorm{\bv} \\
        &\qquad\quad +Ch^2\norm{\partial_t\bu}_{L^\infty(t_m,t_{m+1};[H^1(\Omega)]^d)}\enorm{\bv}+ Ch\tau^{1/2}\norm{\partial_{tt}\bu}_{L^2(t_m,t_{m+1};[L^2(\Omega)]^d)}\enorm{\bv}.
    \end{align*}
   Dividing by \(\enorm{\mathbf v}\), we obtain
    \begin{align}
        \norm{\epsilon_h^{m+1}}_0&\leq C\frac{h}{\tau}\norm{\mathcal{R}(\mathbf{e}_h^{m+1}-\mathbf{e}_h^m)}_0+C\mu\kappa_2\enorm{\mathbf{e}_h^{m+1}}+ C\mu h\norm{\bu^{m+1}}_2\nonumber\\
        &\qquad\quad+Ch^2\norm{\partial_t\bu}_{L^\infty(t_m,t_{m+1};[H^1(\Omega)]^d)}+ Ch\tau^{1/2}\norm{\partial_{tt}\bu}_{L^2(t_m,t_{m+1};[L^2(\Omega)]^d)}.\label{eqn: intermediate_bound_1}
    \end{align}
    Choosing \(\bv=\be_h^{m+1}\) and \(q=\epsilon_h^{m+1}\)
in \eqref{sys: ur_euler_erreqn}
and substituting
\eqref{eqn: ur_euler_erreqn2}
into
\eqref{eqn: ur_euler_erreqn1},
we have
    \begin{align*}
        \frac{1}{\tau}(\mathcal{R}(\mathbf{e}_h^{m+1}-\mathbf{e}_h^m),\mathcal{R}\mathbf{e}_h^{m+1})+&\mu\mathbf{a}(\mathbf{e}_h^{m+1},\mathbf{e}_h^{m+1}) + \mathbf{b}(\bm{\chi}_h^{m+1},\epsilon_h^{m+1})\\
&=l_1(\mathbf{u}^{m+1},\mathbf{e}_h^{m+1})+l_2(\mathbf{u}^{m+1},\mathbf{e}_h^{m+1}) - l^{m}_{\mathtt{BE}}(\mathbf{u},\mathbf{e}_h^{m+1}).
    \end{align*}
    Moreover, \eqref{eqn: supp_conti2} implies
    \begin{equation}
        \mathbf{b}(\bm{\chi}_h^{m+1},\epsilon_h^{m+1})\leq Ch\norm{\epsilon_h^{m+1}}_0\norm{\bu^{m+1}}_2.\label{eqn: intermediate_bound_2}
        \end{equation}
    Using
\eqref{eqn: ba_coer}, \eqref{sys: ur_suppest}, and \eqref{eqn: intermediate_bound_2},
and then multiplying by \(\tau\), we arrive at
    \begin{align*}
        \norm{\cR\be_h^{m+1}}_0^2+\mu\kappa_1\tau\enorm{\be_h^{m+1}}^2&\leq \norm{\cR\be_h^{m+1}}_0\norm{\cR\be_h^{m}}_0+ Ch\tau\norm{\epsilon_h^{m+1}}_0\norm{\bu^{m+1}}_2\\
        &\qquad\quad+C\mu h\tau\norm{\bu^{m+1}}_2\enorm{\be_h^{m+1}}\\
        &\qquad\quad+Ch\tau\norm{\partial_t\bu}_{L^\infty(t_m,t_{m+1};[H^1(\Omega)]^d)}\norm{\cR\be_h^{m+1}}_0\\
        &\qquad\quad+ C\tau^{3/2}\norm{\partial_{tt}\bu}_{L^2(t_m,t_{m+1};[L^2(\Omega)]^d)}\norm{\cR\be_h^{m+1}}_0.
    \end{align*}
   Substituting \eqref{eqn: intermediate_bound_1}
into the previous inequality yields
    \begin{align*}
        \norm{\cR\be_h^{m+1}}_0^2+\mu\kappa_1\tau\enorm{\be_h^{m+1}}^2&\leq \norm{\cR\be_h^{m+1}}_0\norm{\cR\be_h^{m}}_0 + Ch^2\norm{\cR(\be_h^{m+1}-\be_h^{m})}_0\norm{\bu^{m+1}}_2\\
        &\quad\qquad+C\mu\kappa_2 h\tau\norm{\bu^{m+1}}_2\enorm{\be_h^{m+1}} + C\mu h^2\tau\norm{\bu^{m+1}}_2^2\\
        &\quad\qquad +Ch^3\tau\norm{\partial_t\bu}_{L^\infty(t_m,t_{m+1};[H^1(\Omega)]^d)}\norm{\bu^{m+1}}_2\\
        &\quad\qquad + Ch^2\tau^{3/2}\norm{\partial_{tt}\bu}_{L^2(t_m,t_{m+1};[L^2(\Omega)]^d)}\norm{\bu^{m+1}}_2\\
        &\qquad\quad+C\mu h\tau\norm{\bu^{m+1}}_2\enorm{\be_h^{m+1}} \\
        &\qquad\quad+Ch\tau\norm{\partial_t\bu}_{L^\infty(t_m,t_{m+1};[H^1(\Omega)]^d)}\norm{\cR\be_h^{m+1}}_0\\
        &\qquad\quad+ C\tau^{3/2}\norm{\partial_{tt}\bu}_{L^2(t_m,t_{m+1};[L^2(\Omega)]^d)}\norm{\cR\be_h^{m+1}}_0.
    \end{align*}
    Using
\[
\norm{\cR(\be_h^{m+1}-\be_h^m)}_0
\le
\norm{\cR\be_h^{m+1}}_0+\norm{\cR\be_h^m}_0,
\]
 together with Young's inequality,
we obtain
    \begin{align*}
        \norm{\cR\be_h^{m+1}}_0\norm{\cR\be_h^{m}}_0&\leq \frac{1}{2}\norm{\cR\be_h^{m+1}}_0^2 + \frac{1}{2}\norm{\cR\be_h^{m}}_0^2,\\
        h^2\norm{\bu^{m+1}}_2\norm{\cR\be_h^{m+1}}_0&\leq \frac{\tau}{2\alpha}\norm{\cR\be_h^{m+1}}_0^2+\frac{\alpha h^4}{2\tau}\norm{\bu^{m+1}}_2^2,\\
        \mu h\tau\norm{\bu^{m+1}}_2\enorm{\be_h^{m+1}}&\leq \frac{\mu\tau}{2\alpha}\enorm{\be_h^{m+1}}^2+\frac{\alpha \mu h^2\tau}{2}\norm{\bu^{m+1}}_2^2,\\
        h^3\tau\norm{\partial_t\bu}_{L^\infty(t_m,t_{m+1};[H^1(\Omega)]^d)}\norm{\bu^{m+1}}_2&\leq \frac{h^4\tau}{\alpha}\norm{\bu^{m+1}}_2^2+ \frac{\alpha h^2\tau}{2}\norm{\partial_t\bu}_{L^\infty(t_m,t_{m+1};[H^1(\Omega)]^d)}^2,
        \\
        h^2\tau^{3/2}\norm{\partial_{tt}\bu}_{L^2(t_m,t_{m+1};[L^2(\Omega)]^d)}\norm{\bu^{m+1}}_2&\leq \frac{h^4\tau}{\alpha}\norm{\bu^{m+1}}_2^2  + \frac{\alpha \tau^2}{2}\norm{\partial_{tt}\bu}_{L^2(t_m,t_{m+1};[L^2(\Omega)]^d)}^2,
        \\
        h\tau\norm{\partial_t\bu}_{L^\infty(t_m,t_{m+1};[H^1(\Omega)]^d)}\norm{\cR\be_h^{m+1}}_0&\leq  \frac{\tau}{2\alpha}\norm{\cR\be_h^{m+1}}_0^2 + \frac{\alpha h^2\tau}{2}\norm{\partial_t\bu}_{L^\infty(t_m,t_{m+1};[H^1(\Omega)]^d)}^2,\\
        \tau^{3/2}\norm{\partial_{tt}\bu}_{L^2(t_m,t_{m+1};[L^2(\Omega)]^d)}\norm{\cR\be_h^{m+1}}_0&\leq \frac{\tau}{2\alpha}\norm{\cR\be_h^{m+1}}_0^2 + \frac{\alpha \tau^2}{2}\norm{\partial_{tt}\bu}_{L^2(t_m,t_{m+1};[L^2(\Omega)]^d)}^2,
    \end{align*}
    where the constant \(\alpha>0\) may vary from line to line. The remaining terms are treated similarly.
    Consequently, choosing \(\alpha>0\) appropriately yields
    \begin{align*}
        (1-C_0\tau)\norm{\cR\be_h^{m+1}}_0^2 + C\mu\tau\enorm{\be_h^{m+1}}^2&\leq (1+C_1\tau)\norm{\cR\be_h^{m}}_0^2 + C\frac{h^4}{\tau}\norm{\bu^{m+1}}_2^2\\
        &\quad\qquad+C\mu h^2\tau\norm{\bu^{m+1}}_2^2 + Ch^4\tau \norm{\bu^{m+1}}_2^2 \\
        &\qquad\quad + Ch^2\tau\norm{\partial_t\bu}_{L^\infty(t_m,t_{m+1};[H^1(\Omega)]^d)}^2\\
        &\qquad\quad + C\tau^2\norm{\partial_{tt}\bu}_{L^2(t_m,t_{m+1};[L^2(\Omega)]^d)}^2.
    \end{align*}
    Applying a discrete Gr\"onwall's inequality gives
    \begin{align*}
    \|\mathcal{R}\mathbf{e}_h^n\|_0^2 + C\mu \left(\tau\sum_{m=0}^{n-1}\|\mathbf{e}_h^{m+1}\|_\mathcal{E}^2\right)&\leq C^\dagger\|\mathcal{R}\mathbf{e}_h^0\|_0^2 + C^\dagger C\frac{h^4}{\tau^2}\left(\tau\sum_{m=0}^{n-1}\|\mathbf{u}^{m+1}\|_2^2\right)\\
    &\qquad\quad+C^\dagger C(\mu h^2+h^4)\left(\tau\sum_{m=0}^{n-1}\|\mathbf{u}^{m+1}\|_2^2\right)\\
    &\qquad\quad+C^\dagger C h^2\left((t_n-t_0)\norm{\partial_t\bu}_{L^\infty(t_0,t_n;[H^1(\Omega)]^d)}^2\right)\\
    &\qquad\quad+C^\dagger C\tau^2\left(\norm{\partial_{tt}\bu}_{L^2(t_0,t_{n};[L^2(\Omega)]^d)}^2\right),
\end{align*}
    where $C^\dagger>0$ is a constant, independent of $h$, $\tau$, and $n$ for sufficiently small $\tau$, satisfying
\[
\left(\frac{1+C_1\tau}{1-C_0\tau}\right)^n \le C^\dagger
\] 
for all $n\tau\le t_f$.
Absorbing \(C^\dagger\) and \(t_n-t_0\le t_f-t_0\) into the generic constant \(C\), we obtain the desired estimate.
\end{proof}

\begin{remark}[Pressure-robustness]
We emphasize that the right-hand side of the estimate in
Theorem~\ref{thm: euler_error_estimate} contains no pressure
term. This reflects the pressure-robust property of the \texttt{FPR-BE} scheme:
since the pressure never enters the velocity error bound, no inverse
powers of the viscosity parameter $\mu$ appear either. As a result,
the velocity error bound does not deteriorate as $\mu$ becomes small.
\end{remark}

\begin{remark}[Space-time coupling]\label{remark: space-time_coupling}
The estimate in Theorem~\ref{thm: euler_error_estimate} contains the mixed space-time contribution \(h^4/\tau^2\), which corresponds to \(h^2/\tau\) at the level of the velocity error. Similar mesh-time coupling terms appear in fully discrete analyses of related
time-dependent Stokes discretizations \cite{feng2021analysis,lv2024pressure}.
This term shows
that the present estimate does not support arbitrary independent
refinement of \(h\) and \(\tau\). In particular, taking \(\tau\) much
smaller than \(h\) may cause the mixed contribution to dominate the
bound.

To balance the errors arising from the spatial discretization and the backward Euler time stepping, assume that \(\mathbf e_h^0=\mathbf{0}\) and choose \(\tau\simeq h\), i.e., there exist positive constants
\(c_0, c_1\), independent of \(h\) and \(\tau\), such that
\(c_0h\le \tau\le c_1h\).
Then, \(h^2/\tau=\mathcal{O}(h)\), and
Theorem~\ref{thm: euler_error_estimate} implies
\begin{equation}
    \|\cR\be_h^n\|_0 \le C(h + h^2+\sqrt{\mu}h),
    \label{eqn: error_bound_BE_velocity}
\end{equation}
which gives first-order convergence with respect to \(h\).

Alternatively, choosing \(\tau\simeq h^{1/2}\) reduces the mixed contribution to
\(h^2/\tau=O(h^{3/2})\) at the level of the velocity error. However, the first-order temporal consistency error becomes \(O(\tau)=O(h^{1/2})\) and therefore dominates the overall estimate, resulting in a suboptimal \(O(h^{1/2})\) convergence rate with respect to \(h\).
This behavior is verified numerically in Section~\ref{sec:numerical_experiments}.
\end{remark}

\begin{theorem}\label{thm: euler_pressure_error_estimate}
Under the hypotheses of Theorem~\ref{thm: euler_error_estimate}, suppose additionally that \(\mathbf e_h^0=\mathbf{0}\) and \(\tau\simeq h\). Then, for any \(1\le n\le N\),
the pressure error satisfies
\begin{equation*}
\norm{\epsilon_h}_{\ell^2(t_0,t_n;L^2(\Omega))}
\le C(h + h^2 + \sqrt{\mu}h +\mu h + \sqrt{\mu}h^2). 
\end{equation*}
\end{theorem}
\begin{proof}
Squaring both sides of \eqref{eqn: intermediate_bound_1}, multiplying by \(\tau\), and summing over
\(m=0,\dots,n-1\) give
\begin{align}
\tau\sum_{m=0}^{n-1}\norm{\epsilon_h^{m+1}}_0^2
&\leq
C\frac{h^2}{\tau}
\sum_{m=0}^{n-1}
\norm{\mathcal R(\be_h^{m+1}-\be_h^m)}_0^2
+
C\mu^2
\left(\tau\sum_{m=0}^{n-1}\enorm{\be_h^{m+1}}^2\right) \nonumber\\
&\quad\qquad
+
C\mu^2 h^2
\left(
\tau\sum_{m=0}^{n-1}\norm{\bu^{m+1}}_2^2
\right)
+
Ch^2\tau^2
\norm{\partial_{tt}\bu}_{L^2(t_0,t_n;[L^2(\Omega)]^d)}^2 .
\label{eqn: pressure_sum_bound_BE}
\end{align}
Since \(n\tau= t_n-t_0\), we have \(n\le C/\tau\), and the bound \eqref{eqn: error_bound_BE_velocity} implies
\[
\sum_{m=0}^{n-1}
\norm{\mathcal R(\be_h^{m+1}-\be_h^m)}_0^2
\le
\frac{C}{\tau}(h^2+h^4 + \mu h^2).
\]
Hence, using \(\tau\simeq h\), the first term on the right-hand side of \eqref{eqn: pressure_sum_bound_BE} satisfies
\[
C\frac{h^2}{\tau}
\sum_{m=0}^{n-1}
\norm{\mathcal R(\be_h^{m+1}-\be_h^m)}_0^2
\le
C\frac{h^2}{\tau^2}(h^2+h^4 + \mu h^2)\leq C (h^2+ h^4 + \mu h^2 ).
\]
For the second term, Theorem~\ref{thm: euler_error_estimate} gives
\[
C\mu^2\left(\tau\sum_{m=0}^{n-1}\enorm{\be_h^{m+1}}^2\right)\leq C\mu(h^2 + h^4+ \mu h^2 ).
\]
Substituting these estimates into
\eqref{eqn: pressure_sum_bound_BE} yields
\[
\tau\sum_{m=0}^{n-1}
\norm{\epsilon_h^{m+1}}_0^2
\le C(h^2 +h^4 + \mu h^2 + \mu^2h^2 + \mu h^4 ).
\]
Taking the square root completes the proof.
\end{proof}

\begin{remark}
The estimate in Theorem~\ref{thm: euler_pressure_error_estimate} gives
a first-order bound with respect to $h$, with no inverse powers of the viscosity
parameter $\mu$. Hence, the pressure error estimate does not deteriorate
as $\mu$ becomes small. As $\mu$ decreases, the terms $\sqrt{\mu}h$ and
$\mu h$ diminish, leaving the higher-order term $h^2$ to dominate the
bound. This offers a possible explanation for the faster convergence
orders observed for small $\mu$ in Section~\ref{sec:numerical_experiments}.
\end{remark}

\subsection{Error estimates for the FPR-CN method}

For the Crank-Nicolson analysis, we use the midpoint average notation
\[
\bar{\bw}^{m+1/2}:=\frac{\bw^{m+1}+\bw^m}{2}
\]
for any time-dependent quantity \(\bw\), and write
$\bar{\bu}^{m+1/2}$,
$\bar p^{m+1/2}$,
$\bar{\mathbf{f}}^{m+1/2}$,
and $\bar{\bu}_h^{m+1/2}$
for the corresponding averages of the exact and discrete variables.
The averaged error components are defined by
\[
\bar{\bm{\chi}}_h^{m+1/2}
:=
\bar{\bu}^{m+1/2}-\Pi_h\bar{\bu}^{m+1/2},
\qquad
\bar{\be}_h^{m+1/2}
:=
\Pi_h\bar{\bu}^{m+1/2}-\bar{\bu}_h^{m+1/2},
\]
and
\[
\bar{\xi}_h^{m+1/2}
:=
\bar p^{m+1/2}-\mathcal P_0\bar p^{m+1/2},
\qquad
\bar{\epsilon}_h^{m+1/2}
:=
\mathcal P_0\bar p^{m+1/2}-p_h^{m+1/2},
\]
so that
$
\bar{\bu}^{m+1/2}-\bar{\bu}_h^{m+1/2}
=
\bar{\bm{\chi}}_h^{m+1/2}
+
\bar{\be}_h^{m+1/2}$, and
$\bar p^{m+1/2}- p_h^{m+1/2}
=
\bar{\xi}_h^{m+1/2}
+
\bar{\epsilon}_h^{m+1/2}$.

\begin{lemma}\label{lemma: cn_erreqn}
For any \(\bv\in\Vh\) and \(q\in Q_h\), the error components satisfy
\begin{subequations}\label{sys: cn_erreqn}
\begin{alignat}{2}
\frac{1}{\tau}
(\mathcal R(\be_h^{m+1}-\be_h^m),\mathcal R\bv)
+\mu\ba(\bar{\be}_h^{m+1/2}&,\bv)
-\bb(\bv,\bar{\epsilon}_h^{m+1/2})\nonumber\\
&=
l_1(\bar{\bu}^{m+1/2},\bv)
+l_2(\bar{\bu}^{m+1/2},\bv)
-l^{m}_{\mathtt{CN}}(\bu,\bv),
\label{eqn: cn_erreqn1}
\\
\bb(\bar{\be}_h^{m+1/2},q)
&=
-\bb(\bar{\bm{\chi}}_h^{m+1/2},q),
\label{eqn: cn_erreqn2}
\end{alignat}
\end{subequations}
where $l_1(\cdot,\cdot)$ and $l_2(\cdot,\cdot)$ are as defined in Lemma~\ref{lemma: be_fpr_erreqn}, and
\begin{align*}
l^{m}_{\mathtt{CN}}(\bu,\bv)
&:=
\left(\frac{\partial_t\bu^{m+1}+\partial_t\bu^{m}}{2},\mathcal R\bv\right)
-\frac1{\tau}
(\mathcal R\Pi_h(\bu^{m+1}-\bu^m),\mathcal R\bv).
\end{align*}
\end{lemma}
\begin{proof}
We derive the two error equations separately. Testing the momentum equation \eqref{eqn: time_Stokes1} at \(t=t_m\) and \(t=t_{m+1}\)
against \(\mathcal R\bv\) and averaging the two identities, we obtain
\[
\left(\frac{\partial_t\bu^{m+1}+\partial_t\bu^{m}}{2},\mathcal R\bv\right)
-\mu
(
\Delta\bar{\bu}^{m+1/2},
\mathcal R\bv
)
+
(
\nabla \bar{p}^{m+1/2},
\mathcal R\bv
)
=
(\bar{\bbf}^{m+1/2},\mathcal R\bv).
\]
As in the proof of Lemma~\ref{lemma: be_fpr_erreqn}, we have
\[
-\mu(\Delta\bar{\bu}^{m+1/2},\mathcal R\bv)
=
\mu\ba(\Pi_h\bar{\bu}^{m+1/2},\bv)
-
l_1(\bar{\bu}^{m+1/2},\bv)
-
l_2(\bar{\bu}^{m+1/2},\bv),
\]
and
\[
(\nabla \bar{p}^{m+1/2},\mathcal R\bv)
=
-\bb(\bv,\mathcal P_0\bar{p}^{m+1/2}).
\]
For the time derivative, adding and subtracting the reconstructed difference quotient gives
\[
\left(\frac{\partial_t\bu^{m+1}+\partial_t\bu^{m}}{2},\mathcal R\bv\right)
=
\frac1{\tau}
(\mathcal R\Pi_h(\bu^{m+1}-\bu^m),\mathcal R\bv)
+
l^{m}_{\mathtt{CN}}(\bu,\bv).
\]
Substituting these identities into the averaged equation
yields
\[
\begin{aligned}
\frac1{\tau}
(\mathcal R\Pi_h(\bu^{m+1}-\bu^m),\mathcal R\bv)
&+
\mu\ba(\Pi_h\bar{\bu}^{m+1/2},\bv)
-
\bb(\bv,\mathcal P_0\bar{p}^{m+1/2})
\\
&=(\bar{\bbf}^{m+1/2},\mathcal R\bv)
+
l_1(\bar{\bu}^{m+1/2},\bv)
+
l_2(\bar{\bu}^{m+1/2},\bv)
-
l^{m}_{\mathtt{CN}}(\bu,\bv).
\end{aligned}
\]
Subtracting the discrete \texttt{FPR-CN} momentum equation~\eqref{eqn: CN-UR-EG-mom} gives
\eqref{eqn: cn_erreqn1}.

For \eqref{eqn: cn_erreqn2}, since
\({\bu}^{m+1}\) and \({\bu}^{m}\) are both divergence-free, so is $\bar{\bu}^{m+1/2}$, and the discrete continuity equation \eqref{eqn: CN-UR-EG-cont} together with $\bb(\bu_h^m,q)=0$ (which holds inductively from the discrete continuity equation at the previous time level)
implies
\(
\bb(\bar{\bu}^{m+1/2},q)=0
=
\bb(\bar{\bu}_h^{m+1/2},q)\) for all $q\in Q_h$.
Therefore,
\[
\bb(\bar{\be}_h^{m+1/2},q)
=
\bb(\Pi_h\bar{\bu}^{m+1/2}-\bar{\bu}^{m+1/2},q)
=
-\bb(\bar{\bm{\chi}}_h^{m+1/2},q),
\]
which proves \eqref{eqn: cn_erreqn2}.
\end{proof}

\begin{lemma}\label{lemma: cn_time_estimate}
For any \(\bv\in\Vh\), the temporal consistency term satisfies
\begin{equation}
|l^{m}_{\mathtt{CN}}(\bu,\bv)|
\le
C\big(
h\norm{\partial_t\bu}_{L^\infty(t_m,t_{m+1};[H^1(\Omega)]^d)}
+
\tau^{3/2}
\norm{\partial_{ttt}\bu}_{L^2(t_m,t_{m+1};[L^2(\Omega)]^d)}
\big)
\norm{\mathcal R\bv}_0.   \label{eqn: ur_suppest4}
\end{equation}
\end{lemma}
\begin{proof}
By adding and subtracting the standard difference quotient, we decompose
\(l^{m}_{\mathtt{CN}}\) as
\begin{align*}
l^{m}_{\mathtt{CN}}(\bu,\bv) &=  \left(
\frac{\partial_t\bu^{m+1}+\partial_t\bu^m}{2}
-
\frac{\bu^{m+1}-\bu^m}{\tau},
\mathcal R\bv
\right)\\
&\qquad\qquad+\left(
\frac{\bu^{m+1}-\bu^m}{\tau}
-
\frac{\mathcal R\Pi_h(\bu^{m+1}-\bu^m)}{\tau},
\mathcal R\bv
\right)\\
&=:\mathbf{I}_1 + \mathbf{I}_2.
\end{align*}
The estimate for \(\mathbf{I}_2\) follows from the same argument used for the
corresponding interpolation-reconstruction term in the proof of
Lemma~\ref{lemma: be_fpr_consistency}.
It remains to estimate \(\mathbf{I}_1\). By Taylor's theorem with integral
remainder,
\[
\left\|
\frac{\partial_t\bu^{m+1}+\partial_t\bu^m}{2}
-
\frac{\bu^{m+1}-\bu^m}{\tau}
\right\|_0
\le
C\tau^{3/2}
\|\partial_{ttt}\bu\|_{L^2(t_m,t_{m+1};[L^2(\Omega)]^d)}.
\]
Therefore,
\[
|\mathbf{I}_1|
\le
C\tau^{3/2}
\|\partial_{ttt}\bu\|_{L^2(t_m,t_{m+1};[L^2(\Omega)]^d)}
\|\mathcal R\bv\|_0.
\]
Combining this estimate with the bound for \(\mathbf{I}_2\) proves the result.
\end{proof}

\begin{theorem}\label{thm:cn_error_estimate}
Under the regularity assumptions
\eqref{eqn: up_regularity}--\eqref{eqn: time_regularity}
and the additional assumption
\[
\partial_{ttt}\bu\in L^2(I;[L^2(\Omega)]^d),
\]
for any \(1\le n\le N\), the \texttt{FPR-CN} error satisfies
\begin{align*}
    \|\mathcal{R}\mathbf{e}_h^n\|_0^2 + C\mu\left(
\tau\sum_{m=0}^{n-1}
\|\bar{\be}_h^{m+1/2}\|_{\mathcal E}^2
\right) 
    &\leq C\|\mathcal{R}\mathbf{e}_h^0\|_0^2 + C\!\left(\frac{h^4}{\tau^2}+\mu h^2+h^4\right)\!\left(\tau\norm{\bu^0}_2^2+\norm{\bu}^2_{\ell^2(t_0,t_n;[H^2(\Omega)]^d)}\right)\\
    &\qquad\quad+Ch^2\norm{\partial_t\bu}_{L^\infty(t_0,t_n;[H^1(\Omega)]^d)}^2+C\tau^4\norm{\partial_{ttt}\bu}_{L^2(t_0,t_{n};[L^2(\Omega)]^d)}^2.
\end{align*}
\end{theorem}

\begin{proof}
    Following the same argument as in the proof of Theorem~\ref{thm: euler_error_estimate}, we rewrite the pressure term from \eqref{eqn: cn_erreqn1} as
    \begin{align*}
        \bb(\bv,\bar{\epsilon}_h^{m+1/2})&= 
        \frac{1}{\tau}
(\mathcal R(\be_h^{m+1}-\be_h^m),\mathcal R\bv)
+\mu\ba(\bar{\be}_h^{m+1/2},\bv)-
l_1(\bar{\bu}^{m+1/2},\bv)
-l_2(\bar{\bu}^{m+1/2},\bv)
+l^{m}_{\mathtt{CN}}(\bu,\bv).
    \end{align*}
    Applying the Cauchy-Schwarz inequality, the continuity estimate \eqref{eqn: ba_conti}, the supplemental estimates
\eqref{eqn: ur_suppest1}, \eqref{eqn: ur_suppest2}, and \eqref{eqn: ur_suppest4}, together with the discrete
inf-sup condition \eqref{eqn: inf-sup} and the norm bound
\eqref{eqn: norm_bound_RtoE}, we obtain
    \begin{align*}
        \enorm{\bv}\norm{\bar{\epsilon}_h^{m+1/2}}_0&\leq C\frac{h}{\tau}\norm{\mathcal{R}(\mathbf{e}_h^{m+1}-\mathbf{e}_h^m)}_0\enorm{\mathbf{v}}+\mu\kappa_2\enorm{\bar{\mathbf{e}}_h^{m+1/2}}\enorm{\mathbf{v}}+C\mu h\norm{\bar{\bu}^{m+1/2}}_2\enorm{\bv} \\
        &\qquad\quad +Ch^2\norm{\partial_t\bu}_{L^\infty(t_m,t_{m+1};[H^1(\Omega)]^d)}\enorm{\bv}+ Ch\tau^{3/2}\norm{\partial_{ttt}\bu}_{L^2(t_m,t_{m+1};[L^2(\Omega)]^d)}\enorm{\bv}.
    \end{align*}
   Dividing by \(\enorm{\bv}\), we arrive at
    \begin{align}
        \norm{\bar{\epsilon}_h^{m+1/2}}_0&\leq C\frac{h}{\tau}\norm{\mathcal{R}(\mathbf{e}_h^{m+1}-\mathbf{e}_h^m)}_0+C\mu\kappa_2\enorm{\bar{\mathbf{e}}_h^{m+1/2}}+ C\mu h\norm{\bar{\bu}^{m+1/2}}_2\nonumber\\
        &\qquad\quad+Ch^2\norm{\partial_t\bu}_{L^\infty(t_m,t_{m+1};[H^1(\Omega)]^d)}+ Ch\tau^{3/2}\norm{\partial_{ttt}\bu}_{L^2(t_m,t_{m+1};[L^2(\Omega)]^d)}.\label{eqn: cn_intermediate_bound_1}
    \end{align}
    Taking
\(\bv=\bar{\be}_h^{m+1/2}\)
and
\(q=\bar{\epsilon}_h^{m+1/2}\)
in \eqref{sys: cn_erreqn} implies
    \begin{align*}
        \frac{1}{\tau}(\mathcal{R}(\mathbf{e}_h^{m+1}-\mathbf{e}_h^m),\mathcal{R}\bar{\mathbf{e}}_h^{m+1/2})+&\mu\mathbf{a}(\bar{\mathbf{e}}_h^{m+1/2},\bar{\mathbf{e}}_h^{m+1/2}) + \mathbf{b}(\bar{\bm{\chi}}_h^{m+1/2},\bar{\epsilon}_h^{m+1/2})\\
&=l_1(\bar{\mathbf{u}}^{m+1/2},\bar{\mathbf{e}}_h^{m+1/2})+l_2(\bar{\mathbf{u}}^{m+1/2},\bar{\mathbf{e}}_h^{m+1/2}) - l^{m}_{\mathtt{CN}}(\mathbf{u},\bar{\mathbf{e}}_h^{m+1/2}).
    \end{align*}
Furthermore, applying \eqref{eqn: supp_conti2} gives 
    \begin{equation}
        \mathbf{b}(\bar{\bm{\chi}}_h^{m+1/2},\bar{\epsilon}_h^{m+1/2})\leq Ch\norm{\bar{\epsilon}_h^{m+1/2}}_0\norm{\bar{\bu}^{m+1/2}}_2.\label{eqn: cn_intermediate_bound_2}
        \end{equation}
Hence, using the identity
\[
(\mathcal R(\be_h^{m+1}-\be_h^m),\mathcal R\bar{\be}_h^{m+1/2})
=
\frac12
\bigl(
\|\mathcal R\be_h^{m+1}\|_0^2
-
\|\mathcal R\be_h^m\|_0^2
\bigr),
\]
together with the coercivity estimate \eqref{eqn: ba_coer},
the bounds \eqref{eqn: ur_suppest1}, \eqref{eqn: ur_suppest2}, \eqref{eqn: ur_suppest4},
and \eqref{eqn: cn_intermediate_bound_2}, we arrive at
    \begin{align*}
        \norm{\cR\be_h^{m+1}}_0^2+2\mu\kappa_1\tau\enorm{\bar{\be}_h^{m+1/2}}^2&\leq \norm{\cR\be_h^{m}}_0^2  + Ch\tau\norm{\bar{\epsilon}_h^{m+1/2}}_0\norm{\bar{\bu}^{m+1/2}}_2\\
        &\qquad\quad+C\mu h\tau\norm{\bar{\bu}^{m+1/2}}_2\enorm{\bar{\be}_h^{m+1/2}}\\
        &\qquad\quad+Ch\tau\norm{\partial_t\bu}_{L^\infty(t_m,t_{m+1};[H^1(\Omega)]^d)}\norm{\cR\bar{\be}_h^{m+1/2}}_0\\
        &\qquad\quad+ C\tau^{5/2}\norm{\partial_{ttt}\bu}_{L^2(t_m,t_{m+1};[L^2(\Omega)]^d)}\norm{\cR\bar{\be}_h^{m+1/2}}_0.
    \end{align*}
    Substituting \eqref{eqn: cn_intermediate_bound_1} and applying Young's inequality in the same manner as in the proof of
Theorem~\ref{thm: euler_error_estimate}, in particular using
\begin{align*}
    h\tau\norm{\partial_t\bu}_{L^\infty(t_m,t_{m+1};[H^1(\Omega)]^d)}&\norm{\cR\bar{\be}_h^{m+1/2}}_0
    \\
    &\leq \frac{\tau}{4\alpha}(\norm{\cR{\be}_h^{m+1}}_0^2+\norm{\cR{\be}_h^{m}}_0^2)+\frac{\alpha h^2\tau}{2}\norm{\partial_t\bu}_{L^\infty(t_m,t_{m+1};[H^1(\Omega)]^d)}^2,
\end{align*}
and
\begin{align*}
    \tau^{5/2}\norm{\partial_{ttt}\bu}_{L^2(t_m,t_{m+1};[L^2(\Omega)]^d)}&\norm{\cR\bar{\be}_h^{m+1/2}}_0
    \\
    &\leq \frac{\tau}{4\alpha}(\norm{\cR{\be}_h^{m+1}}_0^2+\norm{\cR{\be}_h^{m}}_0^2)+\frac{\alpha \tau^4}{2}\norm{\partial_{ttt}\bu}_{L^2(t_m,t_{m+1};[L^2(\Omega)]^d)}^2,
\end{align*}    
we obtain
    \begin{align*}
        (1-C_2\tau)\norm{\cR\be_h^{m+1}}_0^2 + C\mu\tau\enorm{\bar{\be}_h^{m+1/2}}^2&\leq (1+C_3\tau)\norm{\cR\be_h^{m}}_0^2 +C\frac{h^4}{\tau}\norm{\bar{\bu}^{m+1/2}}_2^2\\
        &\qquad\quad+C\mu h^2\tau\norm{\bar{\bu}^{m+1/2}}_2^2 + Ch^4\tau\norm{\bar{\bu}^{m+1/2}}_2^2 \\
        &\qquad\quad + Ch^2\tau\norm{\partial_t\bu}_{L^\infty(t_m,t_{m+1};[H^1(\Omega)]^d)}^2\\
        &\qquad\quad + C\tau^4\norm{\partial_{ttt}\bu}_{L^2(t_m,t_{m+1};[L^2(\Omega)]^d)}^2.
    \end{align*}
    Therefore, a discrete Gronwall inequality yields
    \begin{align*}
    \|\mathcal{R}\mathbf{e}_h^n\|_0^2 + C\mu \left(\tau\sum_{m=0}^{n-1}\|\bar{\mathbf{e}}_h^{m+1/2}\|_\mathcal{E}^2\right)&\leq C^\ddagger\|\mathcal{R}\mathbf{e}_h^0\|_0^2 + C^\ddagger C\frac{h^4}{\tau^2}\left(\tau\sum_{m=0}^{n-1}\|\bar{\mathbf{u}}^{m+1/2}\|_2^2\right)\\
    &\qquad\quad+C^\ddagger C(\mu h^2+h^4)\left(\tau\sum_{m=0}^{n-1}\|\bar{\mathbf{u}}^{m+1/2}\|_2^2\right)\\
    &\qquad\quad+C^\ddagger Ch^2\left((t_n-t_0)\norm{\partial_t\bu}_{L^\infty(t_0,t_n;[H^1(\Omega)]^d)}^2\right)\\
    &\qquad\quad+C^\ddagger C\tau^4\left(\norm{\partial_{ttt}\bu}_{L^2(t_0,t_{n};[L^2(\Omega)]^d)}^2\right),
\end{align*}
where $C^\ddagger>0$ is a constant satisfying
\[
\left(\frac{1+C_3\tau}{1-C_2\tau}\right)^n\leq C^\ddagger,
\]
and is independent of $h$, $\tau$, and $n$.
Finally, absorbing $C^\ddagger$ and $t_n-t_0$ into the generic $C>0$ and applying the triangle inequality for $\bar{\bu}^{m+1/2} = (\bu^{m+1}+\bu^m)/2$, we obtain the desired estimate.
\end{proof}

\begin{remark}
The estimate in Theorem~\ref{thm:cn_error_estimate} contains the same
mixed space-time contribution \(h^4/\tau^2\) as in the backward Euler
case, corresponding to \(h^2/\tau\) at the level of the velocity error.
However, the Crank-Nicolson temporal consistency error is of order
\(O(\tau^2)\), rather than \(O(\tau)\).

To balance the errors arising from the spatial discretization and the
Crank-Nicolson time stepping, assume that
\(\mathbf e_h^0=\mathbf 0\) and choose
\(\tau\simeq h^{1/2}\). Then,
\(
h^2/\tau=O(h^{3/2})\) and \(\tau^2=O(h)\),
and Theorem~\ref{thm:cn_error_estimate} yields
\begin{equation}
    \|\cR\be_h^n\|_0
    \le
    C(h+h^{3/2}+h^2+\sqrt{\mu}h).
    \label{eqn: error_bound_CN_velocity}
\end{equation}
Thus, the second-order temporal accuracy of the Crank-Nicolson method
allows the larger time-step scaling \(\tau\simeq h^{1/2}\) while
preserving an asymptotically first-order convergence rate for the
reconstructed velocity. This behavior is verified numerically in
Section~\ref{sec:numerical_experiments}.

Moreover, as in the backward Euler case, the estimate contains no
pressure-dependent terms or inverse powers of the viscosity parameter
\(\mu\). Hence, the pressure-robust property of the proposed formulation
is retained without deterioration as \(\mu\) becomes small.
\end{remark}

\begin{theorem}\label{thm: cn_pressure_error_estimate}
Assume the hypotheses of Theorem~\ref{thm:cn_error_estimate}.
In addition, suppose that \(\mathbf e_h^0=\mathbf{0}\) and that the time step
satisfies \(\tau\simeq h^{1/2}\). Then, for any integer \(1\le n\le N\),
the pressure error satisfies
\begin{equation*}
\norm{\bar\epsilon_h}_{\ell^2(t_0,t_n;L^2(\Omega))}
\le
C(
h^{3/2}
+
h^{2}
+
h^{5/2}
+
\sqrt{\mu}h
+
\mu h
+
\sqrt{\mu} h^{3/2}
+\sqrt{\mu}h^2
).
\end{equation*}
\end{theorem}
\begin{proof}
The proof is analogous to that of
Theorem~\ref{thm: euler_pressure_error_estimate}; we only indicate the
main changes. From \eqref{eqn: cn_intermediate_bound_1},
squaring, multiplying by \(\tau\), and summing over
\(m=0,\ldots,n-1\), we obtain
\begin{align*}
\tau\sum_{m=0}^{n-1}
\norm{\bar\epsilon_h^{m+1/2}}_0^2
&\le
C\frac{h^2}{\tau}
\sum_{m=0}^{n-1}
\norm{\mathcal R(\be_h^{m+1}-\be_h^m)}_0^2
+
C\mu^2
\left(
\tau\sum_{m=0}^{n-1}
\enorm{\bar\be_h^{m+1/2}}^2
\right) \\
&\qquad\quad
+
C\mu^2h^2
\norm{\bu}_{\ell^2(t_0,t_n;[H^2(\Omega)]^d)}^2
+
Ch^2\tau^4
\norm{\partial_{ttt}\bu}_{L^2(t_0,t_n;[L^2(\Omega)]^d)}^2 .
\end{align*}
From \eqref{eqn: error_bound_CN_velocity} and using $\tau\simeq h^{1/2}$, we have
\[
C\frac{h^2}{\tau}
\sum_{m=0}^{n-1}
\norm{\mathcal R(\be_h^{m+1}-\be_h^m)}_0^2\leq C\frac{h^2}{\tau^2}(h^2 + h^{3} + h^4 + \mu h^2)\leq C(h^3 + h^4 + h^5 +\mu h^3)
\]
and
\[
C\mu^2
\left(
\tau\sum_{m=0}^{n-1}
\enorm{\bar\be_h^{m+1/2}}^2
\right)
\le
C\mu(h^2 + h^{3} + h^4 + \mu h^2),
\]
while the remaining two terms carry over unchanged. Therefore,
\[
\tau\sum_{m=0}^{n-1}
\norm{\bar\epsilon_h^{m+1/2}}_0^2
\le
C(
h^3
+
 h^4
+
h^5
+\mu h^2
+
\mu^2h^2
+
\mu h^3
+
\mu h^4
).
\]
Taking the square root gives
the desired estimate.
\end{proof}

\begin{remark}
The estimate in Theorem~\ref{thm: cn_pressure_error_estimate} contains
no inverse powers of the viscosity parameter \(\mu\). Hence, the pressure
error estimate does not deteriorate as \(\mu\) becomes small.
For \(\mu=\mathcal O(1)\), the terms \(\sqrt{\mu}h\) and \(\mu h\)
yield a first-order pressure error bound in \(h\). For small values of
\(\mu\), these viscosity-dependent terms are reduced, so
non-viscous contribution \(h^{3/2}\) may become dominant. This explains
the higher observed convergence order reported in
Section~\ref{sec:numerical_experiments}.
\end{remark}

\section{Numerical experiments}
\label{sec:numerical_experiments}

This section presents a series of numerical experiments to assess the performance of the proposed fully pressure-robust schemes. We first consider two-dimensional test problems to examine the role of the velocity reconstruction, verify the predicted convergence behavior and viscosity robustness, and compare the temporal accuracy of the backward Euler and Crank-Nicolson discretizations. We then turn to three-dimensional examples to demonstrate the mass-conservation property of the reconstructed velocity and the applicability of the method to a more practical flow problem.

\subsection{Two-dimensional results}

We begin with a series of two-dimensional numerical experiments designed to examine the role of the velocity reconstruction in the fully discrete formulation. We first compare the proposed fully pressure-robust scheme with standard and partially reconstructed formulations in the small-viscosity regime. This comparison demonstrates that reconstructing only the forcing term is insufficient to prevent deterioration of the velocity approximation and that reconstruction must also be incorporated into the time-derivative term. We then verify the convergence behavior and viscosity-robustness of the proposed backward Euler and Crank-Nicolson schemes and investigate the higher temporal accuracy of the Crank-Nicolson discretization.

\subsubsection{Comparison with standard schemes}
We first compare the proposed fully pressure-robust backward Euler scheme with two standard formulations to examine the role of the reconstruction operator in the fully discrete formulation.
The computational domain is $\Omega=(0,1)^2$, and the time interval is $I=(0,1]$. The exact velocity and pressure are prescribed as
\[
\bu(x,y,t) = \begin{pmatrix}
    \sin(\pi x)\cos(\pi y)\sin(\pi t/2)\\
    -\cos(\pi x)\sin(\pi y)\sin(\pi t/2)
\end{pmatrix},\qquad p(x,y,t) =  \left(\sin(\pi x)\sin(\pi y) - \frac{4}{\pi^2}\right)\sin(\pi t/2).
\]
The velocity is divergence-free, and the constant $4/\pi^2$ is introduced so that the pressure has zero mean over $\Omega$. The forcing term and the boundary and initial data are chosen consistently with the exact solution.

To isolate the effect of the reconstruction in the fully discrete formulation, we compare the following three backward Euler schemes:
\begin{itemize}
    \item \texttt{ST-BE}: the standard backward Euler scheme, in which the reconstruction operator is not applied;
    \item \texttt{SR-BE}: the source-reconstructed backward Euler scheme proposed in \cite{lv2024pressure}, in which the reconstruction is applied only to the forcing term $(\mathbf{f}^{m+1},\cR\bv)_\Omega$;
    \item \texttt{FPR-BE}: the fully pressure-robust backward Euler scheme, in which the reconstruction is applied to both the forcing term and the time-derivative term, as defined in Algorithm~\ref{alg: BE-FPR}.
\end{itemize}
For all three schemes, we choose the small viscosity $\mu=10^{-10}$ and couple the time-step size to the mesh size by setting $\tau=h$.

Table~\ref{table: comparison_st_pr_fpr} reports both the standard $L^2$-velocity error and the reconstructed velocity error, together with the pressure error. The \texttt{ST-BE} scheme exhibits approximately one-half-order convergence for both velocity errors. Applying the reconstruction only to the forcing term substantially reduces the magnitude of these errors, but the \texttt{SR-BE} scheme still converges at approximately the same one-half-order rate. Thus, source-term reconstruction alone reduces the velocity errors but does not recover the expected first-order convergence.
\begin{table}[!ht]
    \centering
    \begin{tabular}{|c|c||c|c|c|c|c|c|}
    \hline
        \multicolumn{8}{|c|}{\texttt{ST-BE} scheme ($\mu=10^{-10}$)} \\
    \hline   
      $h$ & $\tau$  & {\small $ \norm{\mathbf{u}^N-\mathbf{u}_h^N}_0$} & {\small $h$-order} & {\small $ \norm{\cR(\Pi_h\mathbf{u}^N-\mathbf{u}_h^N)}_0$} & {\small $h$-order} & {\small $\norm{\Pz p-p_h}_{\ell^2(0,1;L^2)}$} & {\small $h$-order}\\ 
      \hline
      $1/8$ & 1/8& 4.193e-1 & -  & 4.170e-1 & -  & 1.614e-2 & -\\
      \hline
      $1/16$ & 1/16& 2.865e-1 &  0.55 & 2.860e-1 &  0.54 & 6.642e-3 & 1.28\\
      \hline
      $1/32$ & 1/32& 1.965e-1 &  0.54 & 1.964e-1 & 0.54  & 3.143e-3 & 1.08\\
      \hline
      $1/64$ & 1/64& 1.363e-1 &  0.53 & 1.363e-1 & 0.53  & 1.635e-3 & 0.94 \\
      \hline
      $1/128$ & 1/128& 9.535e-2 &  0.52 & 9.534e-2 &  0.52 & 8.645e-4 & 0.92\\
    \hline
    \hline
          \multicolumn{8}{|c|}{\texttt{SR-BE} scheme ($\mu=10^{-10}$)} \\
    \hline   
      $h$ & $\tau$  & {\small $ \norm{\mathbf{u}^N-\mathbf{u}_h^N}_0$} & {\small $h$-order} & {\small $ \norm{\cR(\Pi_h\mathbf{u}^N-\mathbf{u}_h^N)}_0$} & {\small $h$-order} & {\small $\norm{\Pz p-p_h}_{\ell^2(0,1;L^2)}$} & {\small $h$-order}\\ 
      \hline
      1/8 & 1/8& 1.118e-1 & -  & 1.034e-1 & -  & 3.658e-3 & -\\
      \hline
      $1/16$ & 1/16& 7.347e-2 & 0.61 & 7.169e-2 & 0.53  & 1.195e-3 & 1.61\\
      \hline
      $1/32$ & 1/32& 4.969e-2 &  0.56 & 4.934e-2 &  0.54 & 4.012e-4 & 1.57\\
      \hline
      $1/64$ & 1/64& 3.426e-2 & 0.54 & 3.419e-2 &  0.53 & 1.385e-4 & 1.53\\
      \hline
      $1/128$ & 1/128& 2.390e-2 & 0.52 & 2.389e-2 &  0.52 & 4.893e-5 & 1.50\\
      \hline
      \hline
         \multicolumn{8}{|c|}{\texttt{FPR-BE} scheme ($\mu=10^{-10}$)} \\
    \hline   
      $h$ & $\tau$  & {\small $ \norm{\mathbf{u}^N-\mathbf{u}_h^N}_0$} & {\small $h$-order} & {\small $ \norm{\cR(\Pi_h\mathbf{u}^N-\mathbf{u}_h^N)}_0$} & {\small $h$-order} & {\small $\norm{\Pz p-p_h}_{\ell^2(0,1;L^2)}$} & {\small $h$-order}\\  
      \hline
      $1/8$ & 1/8& 6.719e-2 & -  & 5.027e-2 & -  & 5.029e-4 & -\\
      \hline
      $1/16$ & 1/16& 3.405e-2 & 0.98 & 2.968e-2 &  0.76 & 1.101e-4 & 2.19\\
      \hline
      $1/32$ & 1/32& 1.719e-2 & 0.99  & 1.607e-2 & 0.89  & 2.577e-5 & 2.10\\
      \hline
      $1/64$ & 1/64& 8.635e-3 & 0.99  & 8.354e-3 &  0.94 & 6.250e-6 & 2.04\\
      \hline
      $1/128$ & 1/128& 4.328e-3 &  1.00 & 4.257e-3 & 0.97  & 1.543e-6 & 2.02\\
      \hline
    \end{tabular}
    \caption{Errors and observed convergence orders with respect to $h$ for the \texttt{ST-BE}, \texttt{SR-BE}, and \texttt{FPR-BE} schemes with $\mu=10^{-10}$ and $\tau=h$.}
    \label{table: comparison_st_pr_fpr}
\end{table}
In contrast, the \texttt{FPR-BE} scheme exhibits approximately first-order convergence for both velocity errors, as predicted by the estimate in \eqref{eqn: error_bound_BE_velocity}. These results show that reconstructing the time-derivative term is essential for obtaining a fully pressure-robust discretization of the time-dependent problem. In particular, the comparison between \texttt{SR-BE} and \texttt{FPR-BE} demonstrates that the pressure-robust treatment developed for the stationary problem \cite{hu2024pressure} cannot be extended directly to the time-dependent setting by modifying the forcing term alone. The divergence-free structure must also be incorporated into the discrete temporal evolution.

The pressure errors also exhibit progressively improved convergence behavior across the three schemes. The \texttt{ST-BE} scheme shows approximately first-order convergence, while the \texttt{SR-BE} scheme yields a higher convergence rate. In contrast, the \texttt{FPR-BE} scheme achieves approximately second-order convergence for the pressure error.
These results motivate the use of the \texttt{FPR-BE} and \texttt{FPR-CN} schemes in the remainder of the two-dimensional experiments.

\subsubsection{Convergence and viscosity robustness}\label{subsubsec: pressure_robust}

We now verify the convergence behavior and viscosity robustness of the proposed \texttt{FPR-BE} and \texttt{FPR-CN} schemes using a second manufactured solution.
The computational domain is $\Omega=(0,1)^2$, and the time interval is $I=(0,1]$. The exact velocity and pressure are given by
\[
\bu(x,y,t) = \begin{pmatrix}
    10x^2y(x-1)^2(2y-1)(y-1)t^2\\
    -10xy^2(2x-1)(x-1)(y-1)^2t^2
\end{pmatrix},\qquad p(x,y,t) = 10(2x-1)(2y-1)t^2.
\]
The forcing term and the boundary and initial data are chosen consistently with the exact solution.

We first consider the \texttt{FPR-BE} scheme with the time-step size $\tau=h$. From the velocity and pressure error estimates established in Section~\ref{sec:full_analysis}, we obtain
\begin{subequations}\label{sys: error_estimate_BE}
\begin{alignat}{2}
&\norm{\cR(\Pi_h\bu^n-\bu_h^n)}_0\leq C(h + \sqrt{\mu}h) +\mathcal{O}(h^2),\label{eqn: error_estimate_BE_velo}\\
&\norm{\Pz p - p_h}_{\ell^2(t_0,t_n;L^2(\Omega))}\leq C(h + \sqrt{\mu}h)+ \mathcal{O}(h^2).\label{eqn: error_estimate_BE_pres}
\end{alignat}
\end{subequations}
Thus, both the reconstructed velocity error and the pressure error are expected to converge at least at first order with respect to \(h\), independently of inverse powers of the viscosity.

\begin{table}[!ht]
    \centering
    \begin{tabular}{|c|c||c|c|c|c|}
    \hline
         \multicolumn{6}{|c|}{\texttt{FPR-BE} scheme ($\mu = 1$)
        }  \\
    \hline 
     
      $h$ & $\tau$  & {\small $\norm{\mathcal{R}(\Pi_h\bu^N-\bu_h^N)}_0$}  & {$h$-order} & {\small $\norm{\Pz p-p_h}_{\ell^2(0,1;L^2)}$} & {$h$-order}  \\ 
      \hline
      $1/4$ & 1/4& 5.383e-3 & -  &3.719e-2 & - \\
      \hline
      $1/8$ & 1/8& 1.367e-3 & 1.98  &1.915e-2 & 0.96  \\
      \hline
      $1/16$ & 1/16& 3.371e-4 & 2.02  &8.308e-3 & 1.20\\
      \hline
      $1/32$ & 1/32& 8.676e-5 & 1.96  &3.175e-3 & 1.39 \\
      \hline
      $1/64$ & 1/64& 2.371e-5 & 1.87  &1.314e-3 & 1.27 \\
      \hline
      \hline
      \multicolumn{6}{|c|}{\texttt{FPR-BE} scheme
        ($\mu = 10^{-6}$)}  \\
    
    \hline
     
      $h$ & $\tau$  & {\small $\norm{\mathcal{R}(\Pi_h\bu^N-\bu_h^N)}_0$}  & {$h$-order} & {\small $\norm{\Pz p-p_h}_{\ell^2(0,1;L^2)}$} & {$h$-order}  \\ 
      
      \hline
      $1/4$ & 1/4& 1.900e-2 & -  &5.339e-4 & - \\
      \hline
      $1/8$ & 1/8& 7.665e-3 & 1.31  &7.586e-5 & 2.82  \\
      \hline
      $1/16$ & 1/16& 3.141e-3 & 1.29  &9.300e-6 & 3.03\\
      \hline
      $1/32$ & 1/32& 1.391e-3 & 1.18  &1.043e-6 & 3.16 \\
      \hline
      $1/64$ & 1/64& 6.509e-4 & 1.10  &1.220e-7 & 3.10 \\
      \hline
    \end{tabular}
    \caption{Errors and observed convergence orders with respect to \(h\) for the \texttt{FPR-BE} scheme with \(\tau=h\). Results are shown for \(\mu=1\) and \(\mu=10^{-6}\).}
    \label{table: FPR-BE_test}
\end{table}
Table~\ref{table: FPR-BE_test} reports the errors and observed convergence orders for \(\mu=1\) and \(\mu=10^{-6}\).
For \(\mu=1\), both the reconstructed velocity and pressure errors converge faster than the first-order behavior guaranteed by the estimates in \eqref{sys: error_estimate_BE}.
For $\mu = 10^{-6}$, the reconstructed velocity error converges at approximately first order, in agreement with \eqref{eqn: error_estimate_BE_velo}.
The pressure error exhibits a higher observed convergence rate because the contribution of \(\sqrt{\mu}h\) in \eqref{eqn: error_estimate_BE_pres} becomes negligible and higher-order terms may dominate over the tested meshes.

We next consider the \texttt{FPR-CN} scheme with $\tau=h^{1/2}$. From the corresponding estimates in Section~\ref{sec:full_analysis}, we obtain
\begin{subequations}\label{sys: error_estimate_CN}
\begin{alignat}{2}
&\norm{\cR(\Pi_h\bu^n-\bu_h^n)}_0\leq C(h + \sqrt{\mu}h) +\mathcal{O}(h^{3/2}),\label{eqn: error_estimate_CN_velo}\\
&\norm{\Pz \bar{p} - p_h}_{\ell^2(t_0,t_n;L^2(\Omega))}\leq C(h^{3/2}+\sqrt{\mu}h) + \mathcal{O}(h^2).\label{eqn: error_estimate_CN_pres}
\end{alignat}
\end{subequations}
Therefore, the reconstructed velocity error is again expected to converge at least at first order. For the pressure error, the dominant rate depends on the viscosity. When $\mu$ is of order one, the term $\sqrt{\mu}h$ gives a first-order contribution. When $\mu$ is small, this term becomes negligible, and the higher-order terms may dominate.

\begin{table}[!ht]
    \centering
    \begin{tabular}{|c|c||c|c|c|c|}
    \hline
         \multicolumn{6}{|c|}{\texttt{FPR-CN} scheme ($\mu=1$)
        }  \\
    \hline 
     
      $h$ & $\tau$  & {\small $\| \mathcal{R}(\Pi_h\bu^N-\bu^N_h)\|_0$}  & {$h$-order} & {\small $\norm{\Pz \bar{p}-p_h}_{\ell^2(0,1;L^2)}$} & {$h$-order}  \\ 
      \hline
      $1/4$ & 1/2& 5.390e-3 & -  &3.916e-2 & - \\
      \hline
      $1/16$ & 1/4& 3.333e-4 & 2.01  &6.326e-3 & 1.32\\
      \hline
      $1/64$ & 1/8& 2.152e-5 & 1.98  &1.243e-3 & 1.17 \\
      \hline
      $1/256$ & 1/16& 1.375e-6 &  1.98 &2.883e-4 & 1.05 \\
      \hline
      \hline
         \multicolumn{6}{|c|}{\texttt{FPR-CN} scheme ($\mu = 10^{-6}$)
        }  \\
    \hline 
     
      $h$ & $\tau$  & {\small $\| \mathcal{R}(\Pi_h\bu^N-\bu^N_h)\|_0$}  & {$h$-order} & {\small $\norm{\Pz \bar{p}-p_h}_{\ell^2(0,1;L^2)}$} & {$h$-order}  \\ 
      \hline
      $1/4$ & 1/2& 1.043e-2 & -  &3.141e-2 & - \\
      \hline
      $1/16$ & 1/4& 8.193e-4 & 1.84  &1.937e-3 & 2.01\\
      \hline
      $1/64$ & 1/8& 5.278e-5 & 1.98  &1.213e-4 & 2.00 \\
      \hline
      $1/256$ & 1/16& 3.320e-6 & 2.00  & 7.582e-6 & 2.00 \\
      \hline
    \end{tabular}
    \caption{Errors and observed convergence orders with respect to \(h\) for the \texttt{FPR-CN} scheme with \(\tau=h^{1/2}\). Results are shown for \(\mu=1\) and \(\mu=10^{-6}\).}
    \label{table: FPR-CN_test}
\end{table}
Table~\ref{table: FPR-CN_test} reports the reconstructed velocity and pressure errors for the \texttt{FPR-CN} scheme with $\tau=h^{1/2}$.
For \(\mu=1\), the reconstructed velocity error exhibits approximately second-order convergence, which is higher than the first-order rate guaranteed by the estimate in \eqref{eqn: error_estimate_CN_velo}.
The pressure error initially converges at a rate higher than one, with the observed order gradually approaching first order under mesh refinement.
This behavior is consistent with \eqref{eqn: error_estimate_CN_pres}, since the term $\sqrt{\mu}h$ remains significant when $\mu=1$.
For \(\mu=10^{-6}\), both the reconstructed velocity and pressure errors exhibit approximately second-order convergence.
In this small-viscosity regime, the contribution of the first-order term \(\sqrt{\mu}h\) becomes negligible, allowing the higher-order terms in \eqref{eqn: error_estimate_CN_velo} and \eqref{eqn: error_estimate_CN_pres} to dominate over the tested mesh range.

Finally, we examine the dependence of the numerical errors on the viscosity parameter $\mu$. The mesh size is fixed at $h=1/16$, while $\mu$ is varied from $10^{-2}$ to $10^{-10}$.
\begin{figure}[!ht]
\centering
\includegraphics[width=0.45\linewidth]{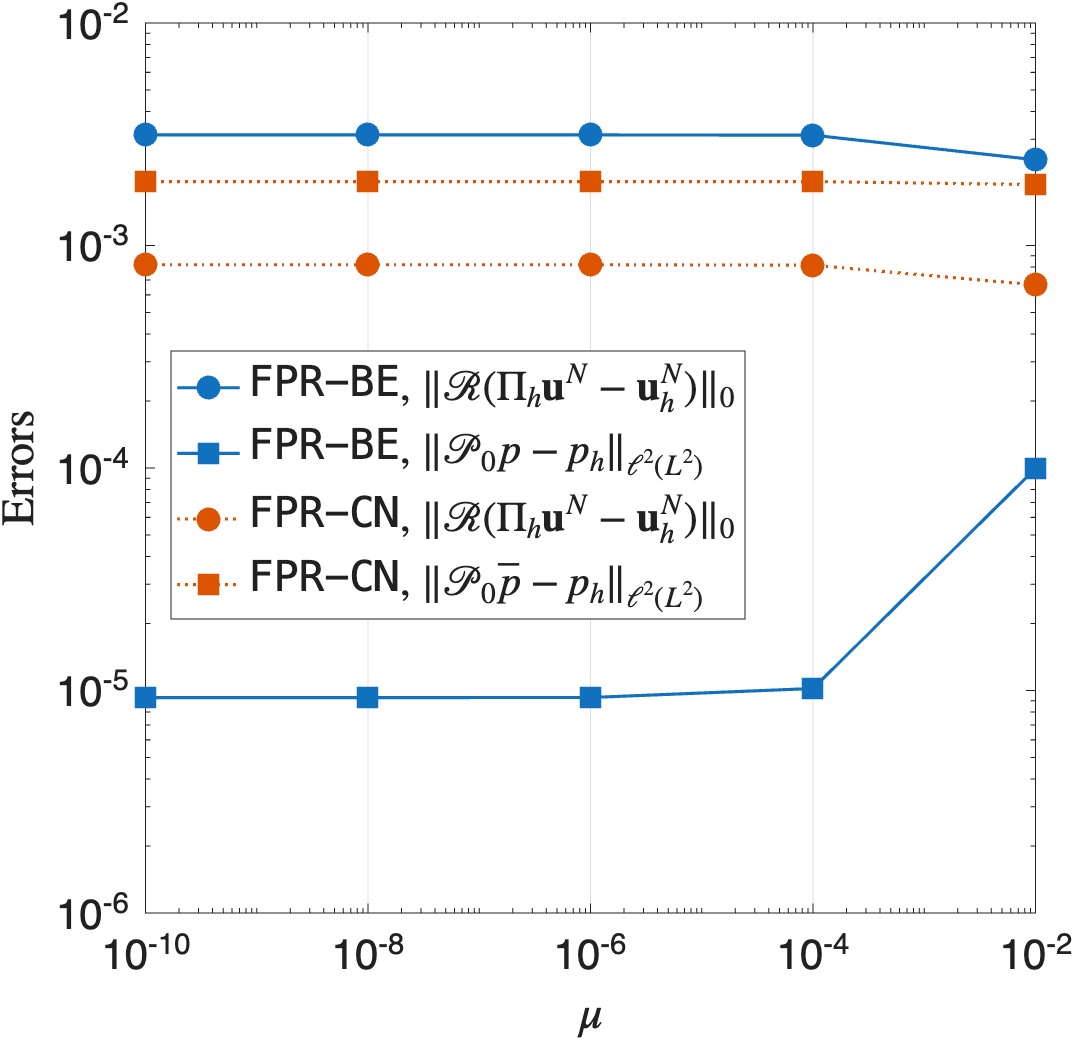}
    \caption{Errors of the \texttt{FPR-BE} and \texttt{FPR-CN} schemes as functions of the viscosity parameter $\mu$, with the mesh size fixed at $h=1/16$.}
    \label{figure: pr}
\end{figure}
As shown in Figure~\ref{figure: pr}, the reconstructed velocity and pressure errors of both schemes remain uniformly bounded and nearly unchanged as $\mu$ approaches zero. This behavior is consistent with the theoretical estimates \eqref{sys: error_estimate_BE} and \eqref{sys: error_estimate_CN}, which contain no inverse powers of $\mu$. The results therefore confirm that the accuracy of the proposed schemes does not deteriorate in the small-viscosity regime.

\subsubsection{Temporal accuracy and pressure initialization}
Using the same exact velocity and pressure fields as in Subsection~\ref{subsubsec: pressure_robust}, we further investigate the effects of the time-step size on the accuracy of the \texttt{FPR-BE} and \texttt{FPR-CN} schemes. Throughout this subsection, the viscosity is fixed at \(\mu=10^{-6}\), and we consider the two choices \(\tau=h\) and \(\tau=h^{1/2}\). These experiments compare the temporal behavior of the two schemes under the same spatial refinement and illustrate the higher temporal accuracy of the Crank-Nicolson discretization.

\begin{table}[!ht]
    \centering
    \begin{tabular}{|c|c||c|c|c|c|}
      \hline
      \multicolumn{6}{|c|}{\texttt{FPR-BE} scheme
        ($\mu = 10^{-6}$)}  \\
    \hline
      $h$ & $\tau$  & {\small $\norm{\mathcal{R}(\Pi_h\bu^N-\bu_h^N)}_0$}  & {$h$-order} & {\small $\norm{\Pz p-p_h}_{\ell^2(0,1;L^2)}$} & {$h$-order}  \\ 
      \hline
      $1/4$ & 1/4& 1.900e-2 & -  &5.339e-4 & - \\
      \hline
      $1/8$ & 1/8& 7.665e-3 & 1.31  &7.586e-5 & 2.82  \\
      \hline
      $1/16$ & 1/16& 3.141e-3 & 1.29  &9.300e-6 & 3.03\\
      \hline
      $1/32$ & 1/32& 1.391e-3 & 1.18  &1.043e-6 & 3.16 \\
      \hline
      $1/64$ & 1/64& 6.509e-4 & 1.10  &1.220e-7 & 3.10 \\
      \hline
      \hline
      \multicolumn{6}{|c|}{\texttt{FPR-CN} scheme
        ($\mu = 10^{-6}$)}  \\ 
    \hline
      $h$ & $\tau$  & {\small $\norm{\mathcal{R}(\Pi_h\bu^N-\bu_h^N)}_0$}  & {$h$-order} & {\small $\norm{\Pz p-p_h}_{\ell^2(0,1;L^2)}$} & {$h$-order}  \\  
      \hline
      $1/4$ & 1/4& 1.043e-2 & -  &3.118e-2 & - \\
      \hline
      $1/8$ & 1/8& 3.119e-3 & 1.74  &7.733e-3 & 2.01  \\
      \hline
      $1/16$ & 1/16& 8.193e-4 & 1.93  &1.936e-3 & 2.00\\
      \hline
      $1/32$ & 1/32& 2.091e-4 & 1.97  &4.848e-4 & 2.00 \\
      \hline
      $1/64$ & 1/64& 5.278e-5 & 1.99  &1.213e-4 & 2.00 \\
      \hline
    \end{tabular}
    \caption{Velocity and pressure errors and their observed convergence orders with respect to $h$ for the \texttt{FPR-BE} and \texttt{FPR-CN} schemes with $\tau=h$ and $\mu=10^{-6}$.}
    \label{table: tau-h_test}
\end{table}
We first set \(\tau=h\). Under this coupling, the theoretical estimates guarantee at least first-order convergence with respect to \(h\) for both schemes. Table~\ref{table: tau-h_test} reports the reconstructed velocity and pressure errors for the \texttt{FPR-BE} and \texttt{FPR-CN} schemes. The velocity error of the \texttt{FPR-BE} scheme approaches first-order convergence, whereas the \texttt{FPR-CN} velocity error converges at a rate close to two. Thus, although the available estimates guarantee only first-order convergence under the choice \(\tau=h\), the Crank-Nicolson scheme exhibits substantially higher velocity accuracy for this smooth test problem.
The pressure error of the \texttt{FPR-CN} scheme is larger in magnitude than that of the \texttt{FPR-BE} scheme on the same mesh, although it converges consistently at approximately second order.
This difference is related to the temporal treatment of the pressure in the \texttt{FPR-CN} formulation. In particular, the discrete pressure \(p_h\) approximates the pressure at the midpoint time level $t_{m+1/2}$ and is compared with \((p^{m+1}+p^m)/2\).

To obtain pressure approximations directly at the time levels \(t_m\), we additionally introduce a pressure-initialized variant of the fully pressure-robust Crank-Nicolson scheme, denoted by \texttt{FPR-CN-PI}. Given \((\mathbf u_h^m,p_h^m)\), the scheme determines \((\mathbf u_h^{m+1},p_h^{m+1})\) by using the previously computed pressure \(p_h^m\) in the Crank-Nicolson momentum equation. In this way, the pressure is advanced from \(t_m\) to \(t_{m+1}\), while the velocity discretization remains identical to that of the original \texttt{FPR-CN} scheme.

\begin{table}[!ht]
    \centering
    \begin{tabular}{|c|c||c|c|c|c|}
      \hline
         \multicolumn{6}{|c|}{\texttt{FPR-BE} scheme ($\mu = 10^{-6}$)
        }  \\
    \hline     
      $h$ & $\tau$  & {\small $\| \mathcal{R}(\Pi_h\bu^N-\bu^N_h)\|_0$}  & {$h$-order} & {\small $\norm{\Pz p-p_h}_{\ell^2(0,1;L^2)}$} & {$h$-order}  \\ 
      \hline
      $1/4$ & 1/2& 2.809e-2 & -  &6.165e-4 & - \\
      \hline
      $1/16$ & 1/4& 1.041e-2 & 0.72  &1.054e-5 & 2.94\\
      \hline
      $1/64$ & 1/8& 4.903e-3 & 0.54  &1.345e-7 & 3.15 \\
      \hline
      $1/256$ & 1/16& 2.433e-3 & 0.51  & 4.358e-9 & 2.47 \\
      \hline
      \hline
         \multicolumn{6}{|c|}{\texttt{FPR-CN} scheme ($\mu = 10^{-6}$)
        }  \\
    \hline 
      $h$ & $\tau$  & {\small $\| \mathcal{R}(\Pi_h\bu^N-\bu^N_h)\|_0$}  & {$h$-order} & {\small $\norm{\Pz \bar{p}-p_h}_{\ell^2(0,1;L^2)}$} & {$h$-order}  \\ 
      \hline
      $1/4$ & 1/2& 1.043e-2 & -  &3.141e-2 & - \\
      \hline
      $1/16$ & 1/4& 8.193e-4 & 1.84  &1.937e-3 & 2.01\\
      \hline
      $1/64$ & 1/8& 5.278e-5 & 1.98  &1.213e-4 & 2.00 \\
      \hline
      $1/256$ & 1/16& 3.320e-6 & 2.00  & 7.582e-6 & 2.00 \\
      \hline
      \hline
         \multicolumn{6}{|c|}{\texttt{FPR-CN-PI} scheme ($\mu = 10^{-6}$)
        }  \\
    \hline 
      $h$ & $\tau$  & {\small $\| \mathcal{R}(\Pi_h\bu^N-\bu^N_h)\|_0$}  & {$h$-order} & {\small $\norm{\Pz p-p_h}_{\ell^2(0,1;L^2)}$} & {$h$-order}  \\ 
      \hline
      $1/4$ & 1/2& 1.043e-2 & -  &6.165e-4 & - \\
      \hline
      $1/16$ & 1/4& 8.193e-4 & 1.84  &1.053e-5 & 2.94\\
      \hline
      $1/64$ & 1/8& 5.278e-5 & 1.98  &1.316e-7 & 3.16 \\
      \hline
      $1/256$ & 1/16& 3.320e-6 & 2.00  & 2.204e-9 & 2.95 \\
      \hline
    \end{tabular}
    \caption{Velocity and pressure errors and their observed convergence orders with respect to $h$ for the \texttt{FPR-BE}, \texttt{FPR-CN}, and \texttt{FPR-CN-PI} schemes with $\tau=h^{1/2}$ and $\mu=10^{-6}$.}
    \label{table: tau-sqrt(h)_test}
\end{table}
We next compare the \texttt{FPR-BE}, \texttt{FPR-CN}, and \texttt{FPR-CN-PI} schemes using the larger time step \(\tau=h^{1/2}\). For the \texttt{FPR-BE} scheme, substituting \(\tau=h^{1/2}\) into the backward Euler error estimate of Theorem~\ref{thm: euler_error_estimate} gives an expected velocity convergence rate of order \(h^{1/2}\), as discussed in Remark~\ref{remark: space-time_coupling}. The results in Table~\ref{table: tau-sqrt(h)_test} agree closely with this prediction, with the observed order approaching \(0.5\) under mesh refinement.
In contrast, both the \texttt{FPR-CN} and \texttt{FPR-CN-PI} schemes retain velocity convergence rates close to two despite the larger time step. 
Their velocity errors are identical up to the reported digits.
The pressure error of the \texttt{FPR-CN-PI} scheme is substantially smaller than that of the original \texttt{FPR-CN} scheme.
These results demonstrate that pressure initialization effectively improves the endpoint pressure approximation without compromising velocity accuracy.

\subsection{Three-dimensional results}

We next present two three-dimensional numerical experiments. The first verifies the mass-conservation property of the reconstructed velocity, while the second demonstrates the performance of the proposed method for flow around a three-dimensional obstacle.

\subsubsection{Verification of mass conservation}
We examine the mass-conservation properties of the \texttt{FPR-BE} and \texttt{FPR-CN} schemes in three dimensions. Let $\Omega = (0,1)^3$, $I = (0,1]$, and $\mu = 10^{-6}$. We consider the exact velocity and pressure
\[
\bu(x,y,z,t) = t^2\begin{pmatrix}
    \sin(\pi x)\cos(\pi y)-\sin(\pi x)\cos(\pi z)\\
    \sin(\pi y)\cos(\pi z)-\sin(\pi y)\cos(\pi x)\\
    \sin(\pi z)\cos(\pi x)-\sin(\pi z)\cos(\pi y)
\end{pmatrix},\qquad
p(x,y,z,t) =  \pi^3\sin(\pi x)\sin(\pi y)\sin(\pi z),
\]
where the velocity is divergence-free, i.e., $\nabla\cdot\bu=0$. 
However, at a given time level $t_m$, the original discrete velocity $\bu_h^m$ is generally not pointwise divergence-free, since the discrete velocity space does not enforce $\norm{\nabla\cdot\bu_h^m}_{0,\Th}=0$.
In contrast, as stated in Remark~\ref{remark: mass_conservation}, the
reconstructed velocity $\mathcal{R}\bu_h^m$ belongs to an
$H(\mathrm{div})$-conforming space and is pointwise divergence-free on
each element. Therefore,
\[
\norm{\nabla\cdot\mathcal{R}\bu_h^m}_{0,\Th}=0
\]
at the discrete level, while in numerical computations this quantity
is expected to remain at the level of machine precision.

Since the exact velocity contains the factor $t^2$, we report the normalized quantities
\[
\left\|\frac{\nabla\cdot\bu_h^m}{t_m^2}\right\|_{0,\Th}
\quad\text{and}
\quad
\left\|\frac{\nabla\cdot\cR\bu_h^m}{t_m^2}\right\|_{0,\Th}.
\]
This normalization removes the prescribed temporal scaling and permits a uniform comparison of the divergence behavior over time. Although such normalization is theoretically unnecessary for $\cR\bu_h^m$, its computed divergence is nonzero at the level of floating-point roundoff. Removing the $t_m^2$ factor prevents the temporal amplitude from affecting these machine-precision values.

\begin{figure}[!ht]
\centering
\includegraphics[width=0.45\linewidth]{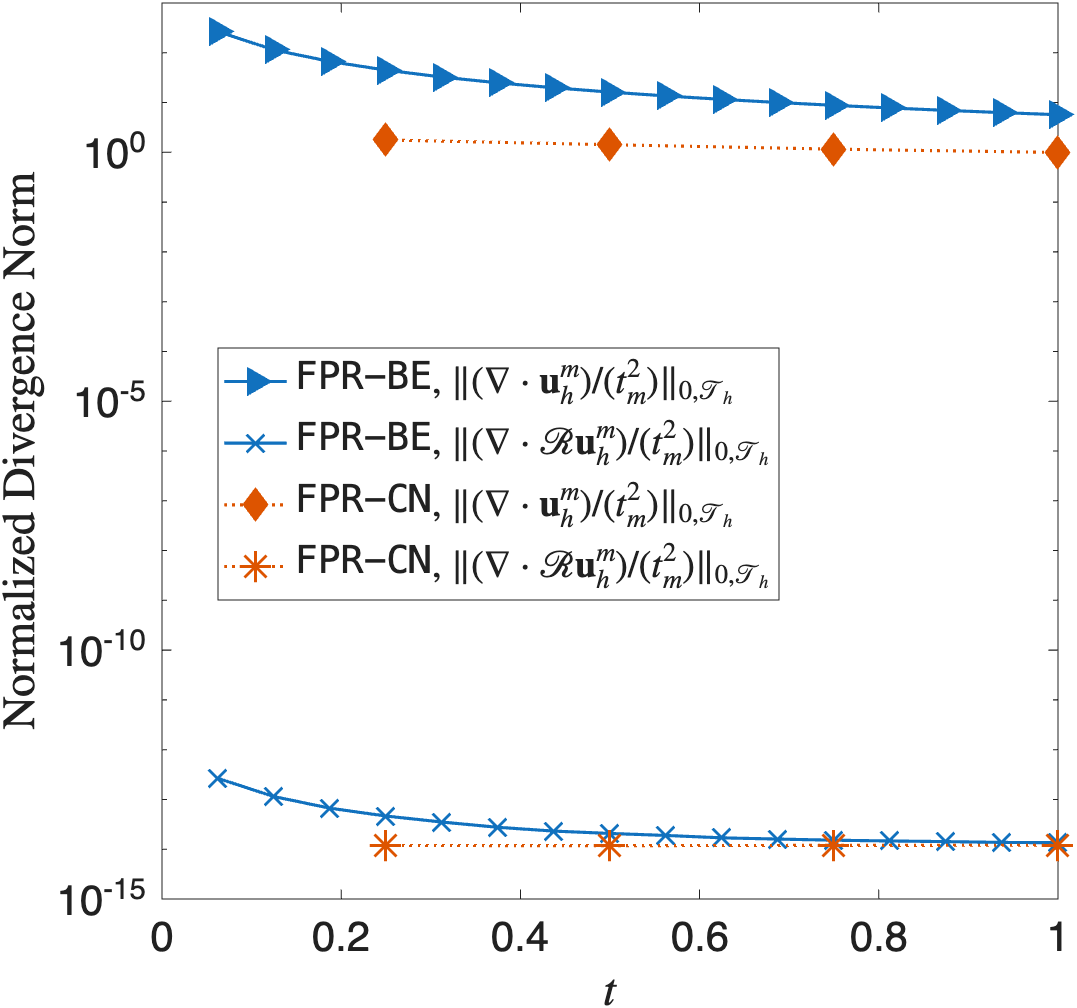}
    \caption{Normalized divergence norms of the original and reconstructed discrete velocities for the three-dimensional \texttt{FPR-BE} and \texttt{FPR-CN} schemes.}
    \label{figure: mass_conservation_3D}
\end{figure}
In Figure~\ref{figure: mass_conservation_3D}, the normalized divergence norm of $\cR\bu_h^m$ remains at approximately $10^{-14}$ for both schemes. This result confirms that $\cR\bu_h^m$ is divergence-free up to machine precision throughout the temporal evolution. The divergence norm of $\bu_h^m$ is nonzero but remains nearly constant in time, indicating that the divergence defect does not accumulate as the solution evolves.
A slight difference can be observed between the two time discretizations: the normalized divergence for the \texttt{FPR-CN} scheme is essentially constant, whereas that of the \texttt{FPR-BE} scheme decreases mildly over time. This behavior is consistent with the Crank-Nicolson method capturing the quadratic temporal factor $t^2$ exactly in this example, whereas backward Euler introduces a first-order temporal error.
Therefore, we conclude that discrete mass conservation is achieved up to machine precision through the reconstructed velocity $\cR\bu_h^m$ in both the \texttt{FPR-BE} and \texttt{FPR-CN} schemes. The behavior of $\bu_h^m$ also provides the supplementary observation that its divergence defect remains stable over time.

\subsubsection{Flow application}

We conclude the numerical experiments with a three-dimensional flow problem in a channel containing a rectangular obstacle. The computational domain is defined by
\[
\Omega = \left[\left(0,\frac32\right)\times(0,1)\times(0,1)\right]\setminus K,
\]
where
\[
K = \left[\frac{7}{16},\frac{9}{16}\right]\times\left[\frac{5}{16},\frac{11}{16}\right]\times\left[\frac{5}{16},\frac{11}{16}\right]
\]
denotes the obstacle. The channel is extended in the $x$-direction to provide sufficient downstream space for the flow to develop after passing the obstacle. The inflow, outflow, obstacle, and wall boundaries are denoted by
\[
\Gamma_{\mathrm{in}} = \{0\}\times (0,1)^2,\quad
\Gamma_{\mathrm{out}} = \left\{\frac{3}{2}\right\}\times(0,1)^2,\quad
\Gamma_{\mathrm{obs}} = \partial K, \quad\text{and}\quad \Gamma_{\mathrm{wall}} = \partial \Omega \setminus (\Gamma_{\mathrm{in}}\cup \Gamma_{\mathrm{out}}\cup \Gamma_{\mathrm{obs}}),
\]
respectively.
We impose the parabolic inflow condition
\[
\bu(0,y,z,t) = \begin{pmatrix}
    36y(1-y)z(1-z)\\
   0\\
    0
\end{pmatrix}\quad \text{on }\Gamma_\mathrm{in},
\]
together with the no-slip condition $\bu=\mathbf{0}$ on $\Gamma_\mathrm{wall}\cup\Gamma_\mathrm{obs}$.
At the outflow boundary, we impose the natural condition
\[
\mu\frac{\partial\bu}{\partial\bn}-p\bn=\mathbf{0}\quad \text{on }\Gamma_\mathrm{out},
\]
and the initial velocity is taken to be $\bu(\cdot,0)=\mathbf{0}$.
We use the \texttt{FPR-CN} scheme with viscosity $\mu=10^{-3}$, mesh size $h=1/16$, and time-step size $\tau = 1/4$.

\begin{figure}[!ht]
\centering
\begin{tabular}{ccc}
Velocity magnitude and streamline ($t=0.5$) && Pressure distribution ($t=0.5$)\\
    \includegraphics[width=.45\textwidth]{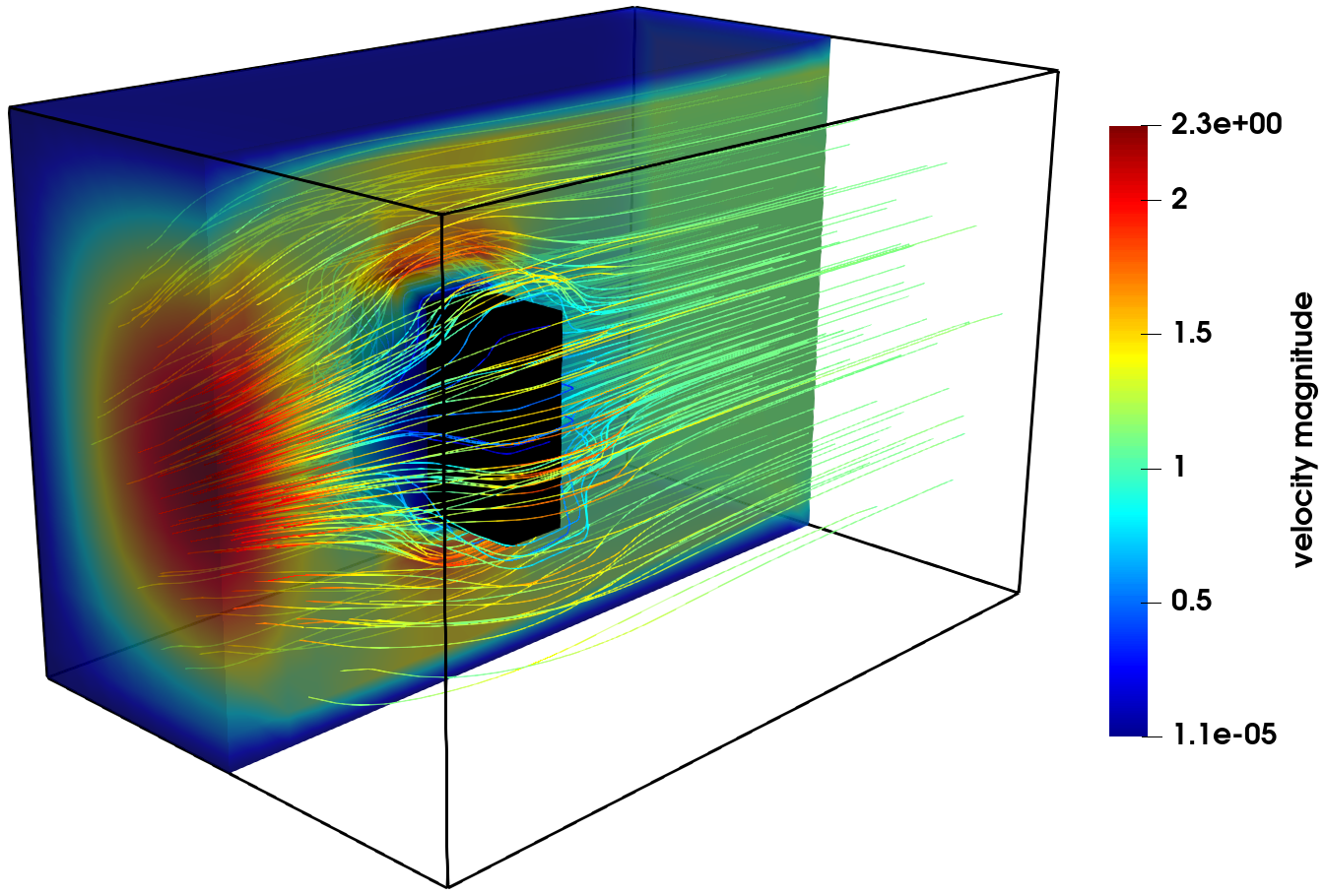}&
   &\includegraphics[width=.45\textwidth]{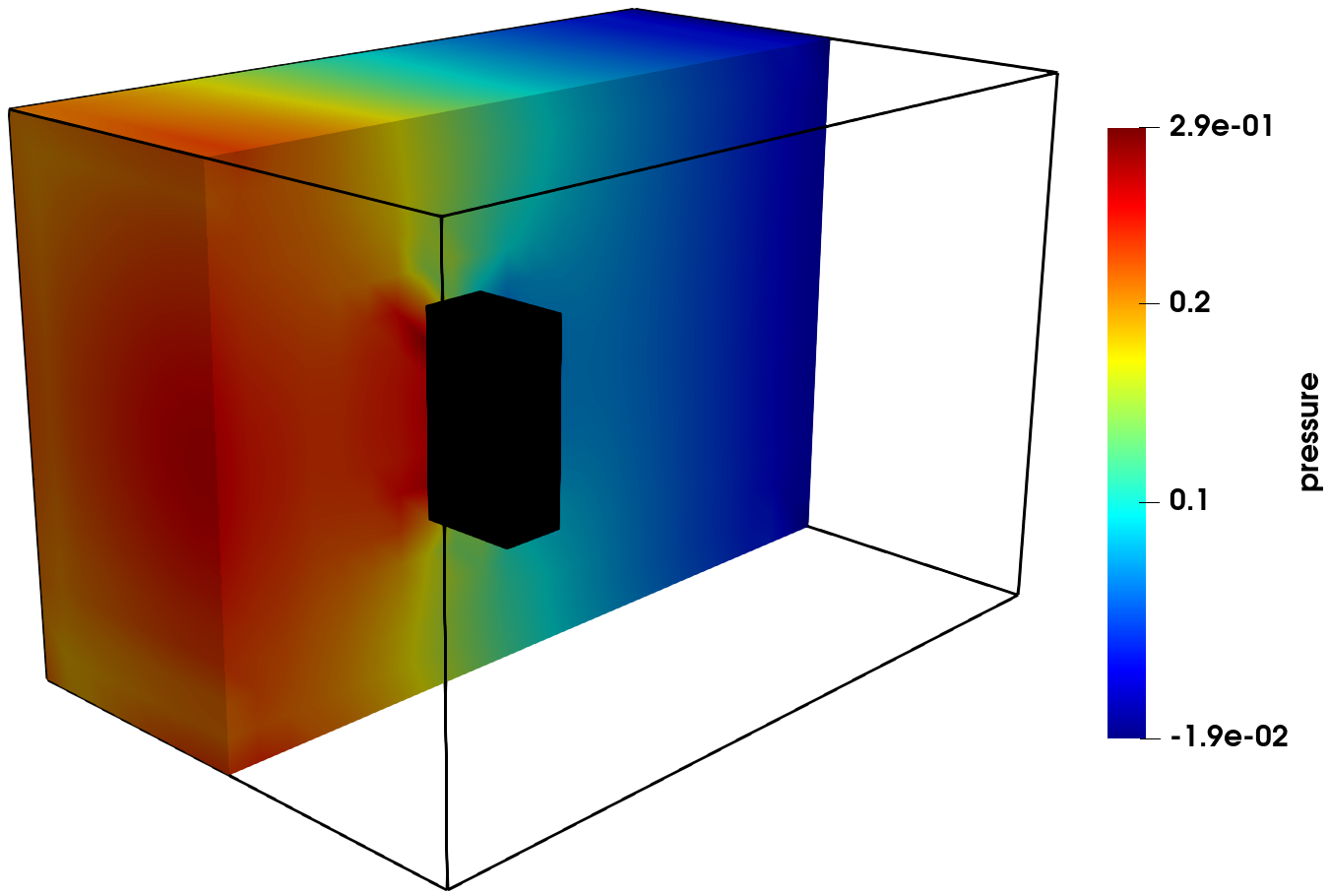}\\
   \\
   Velocity magnitude and streamline ($t=3$) && Pressure distribution ($t=3$)\\
   \includegraphics[width=.45\textwidth]{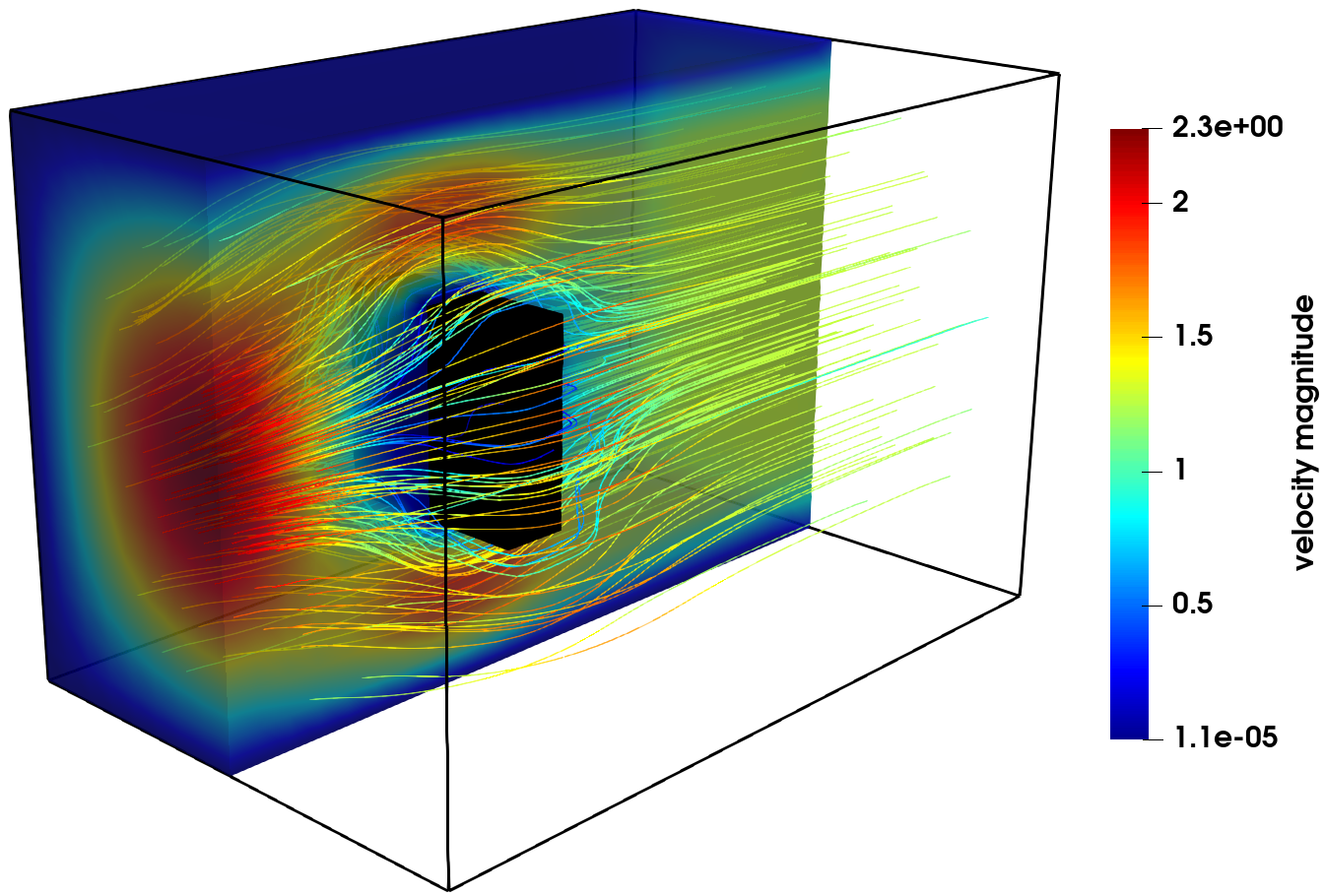}&
   &\includegraphics[width=.45\textwidth]{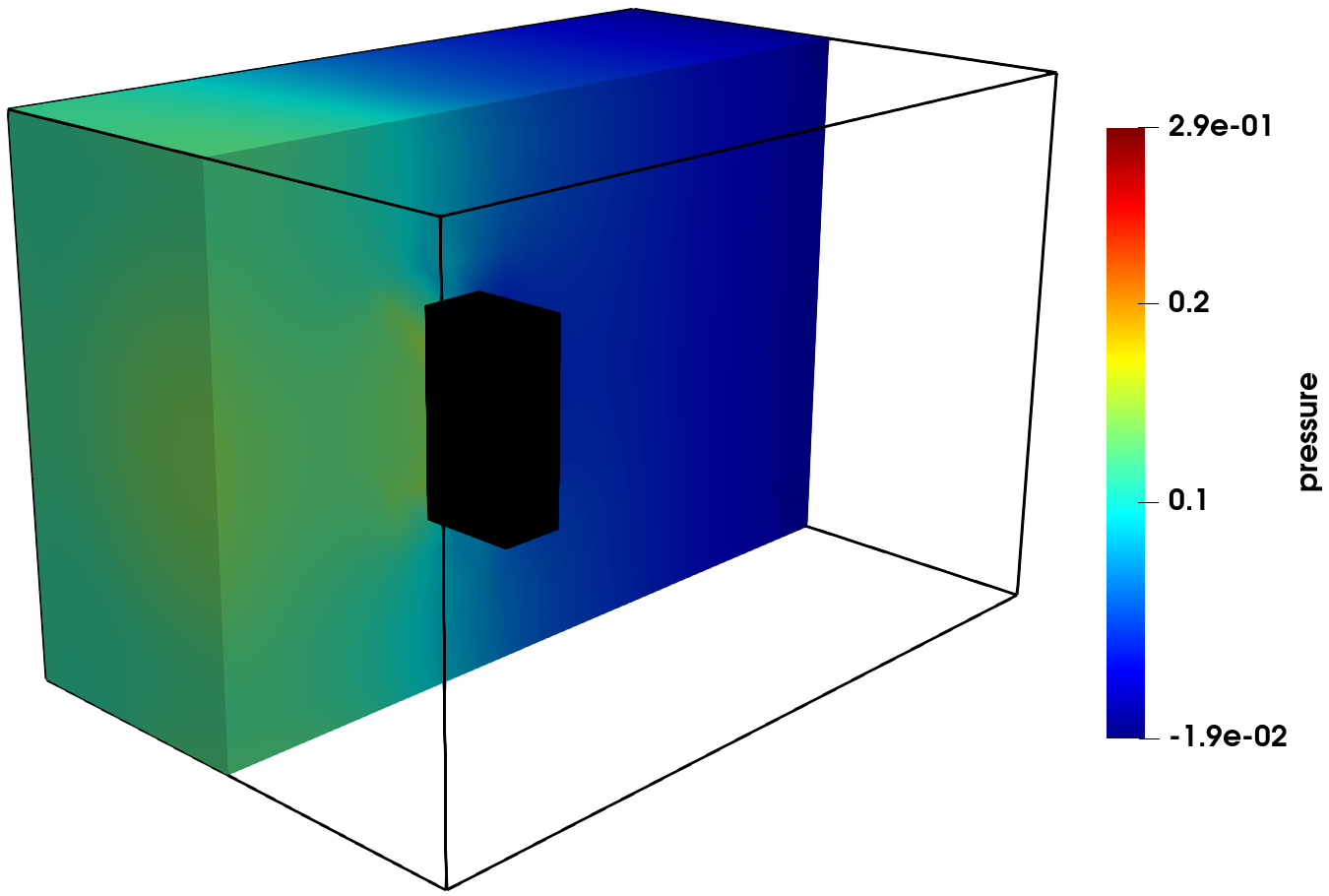}
\end{tabular}
    \caption{Velocity streamlines colored by velocity magnitude and pressure distributions for three-dimensional flow around a rectangular obstacle with $\mu=10^{-3}$ and $h=1/16$. The upper row corresponds to $t=0.5$, and the lower row corresponds to $t=3$.}
    \label{figure: flow_application}
\end{figure}
Figure~\ref{figure: flow_application} presents the velocity and pressure approximations at $t=0.5$ and $t=3$. The velocity streamlines show that the incoming flow separates, passes around the obstacle, and then recombines in the downstream region. The largest velocity magnitudes occur in the regions surrounding the obstacle, where the available flow area is reduced. At $t=3$, a broader high-velocity region is observed around the obstacle than at $t=0.5$, as the flow becomes more fully developed and accelerates through the reduced cross-sectional area surrounding the obstacle. Farther downstream, the streamlines become increasingly aligned with the channel direction, indicating flow recovery after passing the obstacle.
The pressure distribution exhibits the expected upstream-to-downstream decrease. A relatively high-pressure region develops on the upstream side of the obstacle, while a lower-pressure region appears behind it. This pressure difference drives the flow around the obstacle. Between $t=0.5$ and $t=3$, the overall pressure level changes, while the characteristic pressure drop across the obstacle remains clearly visible.
These results demonstrate that the proposed fully pressure-robust formulation can be applied to a three-dimensional flow problem with a nontrivial geometry, mixed boundary conditions, and a small viscosity parameter. The computed fields remain stable and capture the principal flow and pressure features around the obstacle.

\section{Conclusion}\label{sec:conclusion} 

This paper introduced and analyzed pressure-robust, fully discrete enriched Galerkin (EG) methods for the time-dependent Stokes equations.
The spatial discretization employs a tailored velocity reconstruction operator that preserves the continuous component of the EG space while mapping the enriched discontinuous component into the lowest-order Raviart-Thomas ($\mathcal{R}T_0$) subspace of $H(\mathrm{div})$, which ensures strict local mass conservation and pressure-robustness.
The reconstruction is incorporated into both the forcing and discrete time-derivative terms of the backward Euler and Crank-Nicolson schemes, which is essential for retaining pressure-robustness at the fully discrete level.
We established unconditional stability and derived optimal-order, parameter-explicit \textit{a priori} error estimates for the velocity and pressure. In particular, the velocity estimates are independent of the continuous pressure and the irrotational component of the forcing term and contain no inverse-viscosity factors. Numerical experiments in two and three dimensions verified the predicted convergence rates, strict local mass conservation, and robustness to the viscosity parameter. Comparisons with non-pressure-robust variants further demonstrated the importance of incorporating the reconstruction into the discrete time-derivative term.

Future work will extend the proposed framework to nonlinear incompressible flow problems, including the Navier-Stokes equations, with particular attention to preserving pressure-robustness and energy stability in convection-dominated regimes. Higher-order time-stepping methods, such as backward differentiation formulas, will also be investigated to improve temporal accuracy while maintaining pressure- and viscosity-robust fully discrete estimates.

\bibliography{tStokes}
	
\end{document}